\documentclass[a4paper, 10pt]{article}
\usepackage[utf8x]{inputenc}
\usepackage{amsthm,amssymb, amsmath}
\usepackage{graphicx}
\usepackage{subfig}

\newtheorem{theorem}{Theorem}[section]
\newtheorem{definition}[theorem]{Definition}
\newtheorem{proposition}[theorem]{Proposition}

\newtheorem{corollary}[theorem]{Corollary}
\newtheorem{lemma}[theorem]{Lemma}

\newtheorem*{remark}{Remark}

\newcommand{\matr}[4]{
\left( \begin{array}{cc} #1 & #2 \\ #3 & #4 \end{array} \right)}

\newcommand{\RR}{\mathbb{R}}
\newcommand{\QQ}{\mathbb{Q}}
\newcommand{\ZZ}{\mathbb{Z}}
\newcommand{\NN}{\mathbb{N}}

\newcommand{\EE}{\mathcal{E}}
\newcommand{\BB}{\mathcal{B}}

\newcommand{\FF}{\mathcal{F}}
\newcommand{\RRR}{\mathcal{R}}

\title{The bifurcation locus for numbers of bounded type}
\author{ 
C. Carminati\thanks{Dipartimento di Matematica,
Universit\`a di Pisa, Largo Bruno Pontecorvo 5, I-56127, Italy,
email: {\tt carminat@dm.unipi.it}.}\and
G. Tiozzo\thanks{Department of Mathematics, University of Toronto, 40 St George St, Toronto ON, Canada, 
email: {\tt tiozzo@math.utoronto.ca}.}
}

\begin{document}

\maketitle

\begin{abstract}
We define a family $\BB(t)$ of compact subsets of the unit interval which 
provides a  filtration of the set of numbers 
whose continued fraction expansion has bounded digits. 
We study how the set $\BB(t)$ changes as  the parameter $t$ ranges in $[0,1]$, and 
see that the family undergoes period-doubling bifurcations and displays the same transition pattern from
periodic to chaotic behaviour as the family of real quadratic polynomials. 
The set $\EE$ of bifurcation parameters is a fractal set of measure zero and Hausdorff dimension $1$. 
The Hausdorff dimension of $\BB(t)$ varies continuously with the parameter, 
and we show that the dimension of each individual set equals the dimension of the corresponding section 
of the bifurcation set $\EE$.
\end{abstract}

\section{Introduction}
Every $x\in [0,1]\setminus \QQ$ can be encoded in a unique way by its
continued fraction expansion $$x=\cfrac{1}{a_1+\cfrac{1}{a_2+\cfrac{1}{a_3+ \ddots }}}=[0;a_1,a_2,a_3, ...],$$ where
$a_i\in \NN_+$ are called the {\em partial quotients} of $x$ . We say
that $x$ is a {\em number of bounded type} if its partial quotients are bounded.

More precisely, given $N\in \NN$ one can consider the set of numbers whose partial quotients are bounded by $N$:
\begin{equation}\label{eq:bt}
 \mathcal{B}_N:= \{ x = [0; a_1, \dots, a_k, \dots] \ : \ 1 \leq a_k \leq N \ \ \ \ \forall k \geq 1 \}.
\end{equation}

The sets $\BB_N$ have been studied by several authors (see \cite{He2} and the references therein for an account)
and their Hausdorff dimensions have been computed very precisely \cite{JP, JP2}. Moreover,
numbers of bounded type play a key role, not only in number theory, but also
in various contexts in dynamics, for instance in the theory of linearization
of analytic maps, and KAM theory.

Let us recall that the Gauss map, defined as $G(x) := \left\{ \frac{1}{x} \right\}$ for $x > 0$ and $G(0) := 0$,  acts as a shift operator 
on the continued fraction expansion, hence notice that $\mathcal{B}_N$ can be described in terms 
of the orbit for the Gauss map as

$$\mathcal{B}_N = \left\{ x \in [0, 1] \ : \ G^k(x) \geq \frac{1}{N+1} \ \ \ \ \forall k \in \NN \right\},$$
where in this paper $\NN = \{ 0, 1, 2, \dots \}$.
The sets $\BB_N$ give a filtration of the set of numbers of bounded type $\BB :=\bigcup_{N\geq 1} \BB_N$. 
More generally, if we consider sets of points whose orbit
stays above a certain fixed threshold $t$, we obtain a filtration of
$\BB$ with a continuous parameter which refines $\BB_N$:
\begin{definition}
Given $t \in [0,1]$, we define the set $\BB(t)$ of \emph{numbers of type bounded by }t as
\begin{equation}\label{eq:fareyway}
\begin{array}{l}
\BB(t):=\{x \in [0,1] \ : \ G^k(x) \geq t \ \ \ \forall k \in \NN\}.\\
\end{array}
\end{equation}
\end{definition}

Clearly, $\BB_N = \BB\left(\frac{1}{N+1} \right)$, and $\BB(t)
\subseteq \BB(t')$ if $t \geq t'$. The union of all $\BB(t)$ is
the set of all numbers of bounded type $\BB=\bigcup_{t>0}\BB(t)$.
Every $\BB(t)$ is a compact subset of the unit interval, and it is
nowhere dense of measure zero whenever $t > 0$.  

\medskip
\noindent {\bf Examples:}
\begin{enumerate}
\item Let $g:=[0;\overline{1}]=\cfrac{\sqrt{5}-1}{2}$, it is easy to
  see that $\BB(g)=\{g\}$, in fact $\BB(t)=\{g\} $ for all $t$ between
  $[0;\overline{2}]$ and $g$, while $\BB(t)=\emptyset$ for all $t>g$;
\item for $\alpha=[0;\overline{2,1}]$, $\BB(\alpha)=\BB_2=\BB(1/3)$ is
  the set of numbers whose continued fraction contains only the digits
  $1$ and $2$;
\item for $t=g^2=[0;2,\overline{1}]$ we get that $\BB(g^2)$ is the set
  of numbers whose continued fraction contains only the digits $1$ and
  $2$ and such that between any two $2$s there is an even number of
  $1$s. In formulas, $\BB(g^2) = A_0\cup \phi{A_0}$ with
  $A_0:=\{x=[0;X], \ \ X\in \{(2),(1,1)\}^\NN \}$ and
  $\phi(x):=1/(1+x)$.
\end{enumerate}

We will study how the set $\BB(t)$ changes as the parameter
$t$ varies. We will see that our family of sets undergoes period-doubling
bifurcations, and it displays the same transition pattern from
periodic to chaotic behaviour as the family of real quadratic
polynomials.

\medskip
\noindent \textbf{Topological bifurcations.} Let us define the \emph{bifurcation locus} of a function of the unit
interval as the complement of the set of points where the function is
locally constant. The first result (Section \ref{biflocus}) is the

\begin{proposition}\label{prima}
The bifurcation locus for the set-valued function $t \mapsto \BB(t)$ 
is exactly the set
\begin{equation}\label{eq:EE}
\EE := \{ x \in [0, 1] \ : \ G^k(x) \geq x \ \forall k \in \NN \}.
\end{equation}
\end{proposition}

The bifurcation set $\EE$ appears in several, seemingly unrelated
contexts: it was introduced in \cite{CT} to describe the set of phase
transitions of the entropy of $\alpha$-continued fraction
transformations, and it is essentially the same object as the spectrum
of recurrence quotients for cutting sequences of geodesics on the
torus studied by Cassaigne (\cite{Ca}, Theorem 1.1, see also Section \ref{Cassaigne}).

The set $\EE$ is also related to kneading sequences of unimodal maps:
indeed, in \cite{BCIT}  the authors observe that the set $\EE$ can be also expressed using the Farey map $F$
\begin{equation}\label{eq:F}
\EE = \{ x \in [0, 1] \ : \ F^k(x) \geq x \ \forall k \in \NN \}
\end{equation}
and they show that, by means of the conjugation of the Farey map with the tent map, one can build an explicit dictionary between $\EE$ and the set of external rays landing
on the real section of the boundary of the Mandelbrot set
$\mathcal{M}$.  In this sense, this paper complements the one just
mentioned, since our sets $\BB(t)$ are the formal analogue to the
family of Julia sets for real quadratic polynomials $p_c(z)=z^2 +
c$. More precisely, under the same dictionary as in \cite{BCIT}, one
can pass from each $t\in \EE$ to a definite point $c\in
\partial\mathcal{M}\cap \RR$, in such a way that $\BB(t)$ corresponds
to the set of external rays landing on the \emph{Hubbard tree} of the
corresponding real quadratic polynomial $p_c$ (see \cite{Ti-thesis}, Theorem 1.4).

The sets $\BB_N=\BB(1/(N+1))$ are Cantor sets for all $N\geq 2$,
but this is not always the case in the family $\BB(t)$, and indeed the topology 
changes wildly with the parameter. 

However, we are able to to describe each $\BB(t)$ by determining the connected components of its complement.
In order to do so, let $t \in (0, g]$, and define $\ell(t) := \min \left(  \EE \cap [t, 1] \right)$. We say that a rational $r$ is a \emph{$t$-label} 
if there exists $m \geq 1$ with  
$G^m(r)=0$
and 
$G^k(r) > \ell(t)$ for all  $0\leq k \leq m-1$. 
We shall denote by $\QQ_t$ the set of all $t$-labels and $S\cdot x$ the number whose continued fraction expansion is the concatenation of $S$ and the expansion of $x $  (see \textsection \ref{notation}).

\begin{proposition} \label{anatomyprop-intro}
For any $t \in [0, g]$, the connected components of the complement of $\BB(t)$ are 
in bijection with 
$t$-labels; indeed, we have 
$$\BB(t) = [0, 1] \setminus \bigcup_{r \in \mathbb{Q}_t} V_r$$
where 
$V_r:=(S_1 \cdot \ell(t), S_0 \cdot \ell(t))$ and $r$ has continued fraction expansions $r = [0; S_0] = [0; S_1]$, 
where $S_0$ and $S_1$ have, respectively, even and odd length.
Note that different $V_r$ for different $r \in \mathbb{Q}_t$ are disjoint.
\end{proposition}

Some special care is needed for $r = 0, 1$ (for details, see Section \ref{complementary}). For example for $t\leq \sqrt{2}-1$ the rational $r=1/2$ is a $t$-label, $S_0 = (1, 1)$, $S_1 = (2)$ and the 
endpoints of 
$V_{1/2}$ are $$\frac{1}{2+\ell(t)} \ \ \ \mbox{ and } \ \ \  \cfrac{1}{1+\cfrac{1}{1+\ell(t)}}.$$
As a consequence of Proposition \ref{anatomyprop-intro}, we will characterize the set of values such that $\BB(t)$ contains isolated points:

\begin{corollary} \label{isolated-intro}
The set $\BB(t)$ contains isolated points if and only if $\ell(t)$ is isolated in $\EE$.
\end{corollary}

In fact, isolated points of $\mathcal{E}$ can be completely characterized, and are the union of countably many \emph{period-doubling cascades} (see Proposition \ref{P:pd}). 

For instance, the first (largest) period-doubling cascade $(\alpha_n)_{n \geq 0}$ is given by 
$$\alpha_0 = [0;  \overline{1, 1}], \qquad \alpha_1 = [0;  \overline{2}], \qquad \alpha_2 = [0;  \overline{2, 1, 1}], \qquad \alpha_3 = [0;  \overline{2, 1, 1, 2, 2}], \dots$$
and their accumulation point is 
$c_F = [0; 2, 1, 1, 2, 2, 2, 1, 1, 2, 1, 1, \dots]$, whose continued fraction expansion is the unique fixed point of the 
substitution 
$$\left\{ \begin{array}{ll}
1 \mapsto 2 \\
2 \mapsto 2 1 1.
\end{array} \right.$$  

Finally, we prove that isolated points of $\mathcal{E}$
are exactly the points of
discontinuity for the function $t \mapsto \BB(t)$, where we consider the Hausdorff topology on compact sets:
  
\begin{corollary}
The points of discontinuity of the set-valued function
$$t \mapsto \BB(t)$$ in the Hausdorff topology are precisely the
isolated points of $\EE$.
\end{corollary}

\medskip
\noindent \textbf{Hausdorff dimension.}
Another interesting way to study the variation in structure of the family $\BB(t)$ consists in 
considering the Hausdorff dimension $\eta(t) := \textup{H.dim }\BB(t)$ as a function of the parameter $t$
(see Figure 1 for a picture). Such function is locally constant almost everywhere, and displays a
devil's staircase behaviour: it is a singular (but  continuous) function.

\begin{figure}[h]
  \centering
  \subfloat{\includegraphics[width=0.45\textwidth]{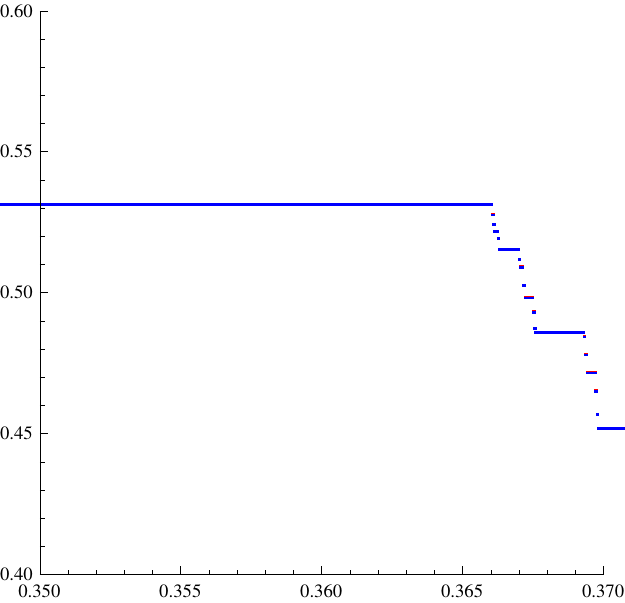}}
\quad
  \subfloat{\includegraphics[width=0.45\textwidth]{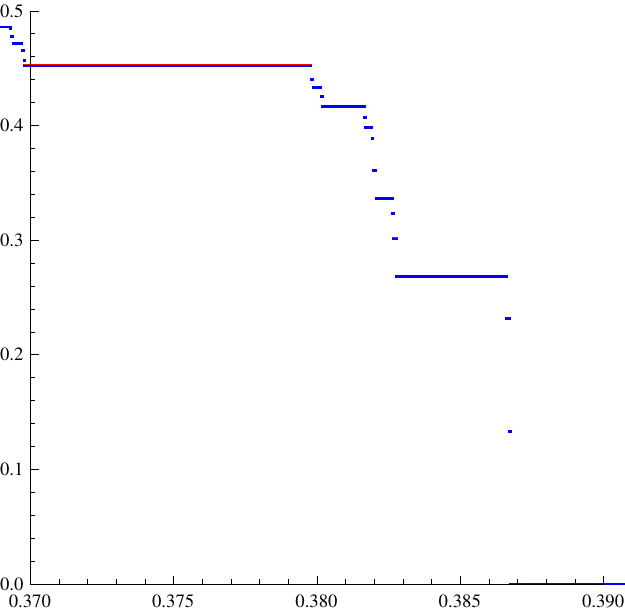}}
\caption{Hausdorff dimension of $\BB(t)$ for $0.35 \leq t \leq 0.37$ (left) and for $0.37 \leq t \leq 0.39$ (right).
The dimension drops to $0$ around the ``Feigenbaum value'' $c_F = 0.3867499707...$ (see Section \ref{extrat}).}
\end{figure}

Pushing further the comparison with Julia sets, let us recall that also in the quadratic polynomial case
one can consider the dependence of Hausdorff dimension of Julia sets on the parameter, 
or the dependence of topological entropy of unimodal maps on the kneading parameter.
In the latter case, which turns out to be much more similar to ours, entropy varies continuously, 
and its plateaux have been characterized in \cite{Do2}.

Moreover, we will relate the Hausdorff dimension of each $\BB(t)$ to the dimension
of a section of parameter space. To state the theorem, we define a rational number $r \in \mathbb{Q} \cap (0, 1]$ to be \emph{extremal} if there exists $n \geq 1$ such that $G^n(r) = 0$ and $G^k(r) > r$ for any $1 \leq k \leq n-1$. Given the two continued fraction expansions   $r = [0; S_0] = [0; S_1]$, we define the \emph{tuning window} (see \textsection \ref{tuning}) associated to $r$ as the interval 
$$W_r := [\omega, \alpha)$$
whose endpoints are the quadratic irrationals $\omega := [0; S_1 \overline{S_0}]$ and $\alpha := [0; \overline{S_0}]$.
For instance, if $r = \frac{1}{3} = [0; 3] = [0; 2,1]$, then $\omega = [0; 3, \overline{2, 1}] = \frac{5 - \sqrt{3}}{11}$, $\alpha = [0; \overline{2, 1}] = \frac{\sqrt{3}-1}{2}$.
\newline

\begin{theorem} \label{main}
\label{continuous}
Let $\eta(t):=  \textup{H.dim }\BB(t).$
Then:
\begin{enumerate}
\item
The function $\eta(t)$ is continuous.
\item
Denote $\EE(t) := \EE \cap [t, 1]$. Then for any $t \in [0,1]$ we have 
$$\textup{H.dim }\BB(t) = \textup{H.dim } \EE(t).$$
\item
The function $\eta(t)$ is constant on any tuning window $W_r$ with $r \in \mathbb{Q} \cap (0, c_F)$ extremal.
\end{enumerate}
\end{theorem}

As a corollary of (3), together with Proposition \ref{prima}, it follows that the bifurcation locus for the set-valued function $t \mapsto \BB(t)$ 
is strictly larger than the bifurcation locus for the dimension function $\eta(t)$. Indeed, the intersection of any tuning window $W_r$ with $\mathcal{E}$ is an uncountable set. 

We conjecture that the statement of Theorem \ref{main}(3) is sharp, in the sense that every interval where the dimension function $\eta(t)$ is constant is contained in the closure $\overline{W_r}$ of a tuning window with $r < c_F$. 

Theorem \ref{main} is formally similar to a widely known result of Urba{\'n}ski
\cite{Ur} for expanding circle maps. Our methods, however, are completely different: we  do  not use any thermodynamic formalism, but  a  renormalization scheme.

Note moreover that the geometry and combinatorics of the bifurcation locus is different, since in 
Urba{\'n}ski's case (in which the map is orientation preserving) the bifurcation locus for the set-valued function is a Cantor set and it coincides 
with the bifurcation locus for the dimension (see \cite{CT3}), while neither statement is true in our case
because of isolated points and tuning windows.

More generally, both of these cases are examples of open dynamical systems: we refer to \cite{PU} for several more 
recent developments in the field.

\medskip
\noindent \textbf{Tuning operators.} 
The first claim of Theorem \ref{continuous} follows quite easily from a general result of Raith \cite{Ra}, while the proof of 
the other two claims rely on renormalization techniques which were introduced in \cite{CT2}
and are illustrated in Section \ref{tuning}. 

Namely, given a rational number $r$, we will define a \emph{tuning map} $\tau_r$
which maps parameter space into a copy of itself; this tuning is
a natural translation, via the dictionary of \cite{BCIT}, of the classical
Douady-Hubbard tuning for quadratic polynomials (\cite{Do1}, \cite{DH}). 
Each tuning map $\tau_r$ maps the unit interval into the tuning window $W_r$; the set $UT$ of \emph{untuned parameters}
is the complement of the union of all tuning windows\footnote{Since tuning is the inverse of renormalization, untuned parameters 
can also be called \emph{non-renormalizable}.}.
The proof consists of three steps: 
\begin{enumerate}
\item Since the functions $x \mapsto \BB(x)$ and $x \mapsto \EE(x)$ are both locally constant on the complement on $\EE$, it is enough to prove 
$$\textup{H.dim }\BB(x) = \textup{H.dim }\EE(x) \quad \textup{for all }x \in \mathcal{E}.$$

\item We first consider the case (Section \ref{dominant}) when $x$ is an untuned parameter: in that case, one proves that there is almost an embedding of 
$\BB(x)$ in $\EE(x)$, or more precisely  

\begin{proposition} \label{Lipschitz}
Given any point $x \in UT$ and any $y > x$, then $\mathcal{E}(x)$
contains a Lipschitz image of $\mathcal{B}(y)$.
\end{proposition}

Equality of dimensions then follows from continuity (Theorem \ref{continuous}).

\item We proceed by induction, using renormalization. Indeed, we prove formulas which relate the dimensions of $\EE(x)$ and $\BB(x)$ 
to the dimensions of their tuned copies, and using the untuned case as the base step we prove equality of dimensions 
for all parameters which are finitely-renormalizable. Such parameters are dense in $\EE$, hence the general claim 
again follows from continuity. 
\end{enumerate}

Let us remark that the strategy outlined above can be applied in several related settings; in particular,
they have been applied, on the continued fraction side, to identify plateaux for the entropy of $\alpha$-continued fractions \cite{CT2}, 
while, on the real quadratic polynomial side, to prove an analogue of Theorem \ref{main} (2) about the real sections of the Mandelbrot set \cite{Ti2}.

The renormalization approach is also useful to study the modulus of continuity of $\eta(t)$, 
which turns out not to be H\"older-continuous at the fixed point $c_F$ of the tuning operator $\tau_{1/2}$, 
which is also the least zero value of $\eta(t)$
(see Section \ref{feigenbaum}). Indeed, we have 
the following precise scaling law:

\begin{theorem}\label{th:scaling}
There exist constants $c_1, c_2 > 0$ such that for any $t < c_F$
\begin{equation} \label{modulus}
\frac{c_1}{-\log |t - c_F|} \leq \textup{H.dim }\BB(t) \leq \frac{c_2}{-\log |t - c_F|}. 
\end{equation}
As a consequence, the function $\eta(t)$ is not locally H\"older continuous in a neighborhood of $c_F$.

Moreover, if $\alpha_n := \tau_{1/2}^n(\alpha)$ with $\alpha =[0;\overline{2,1} ]$, then 
\begin{equation} \label{scaling law-intro}
\lim_{n \to +\infty} \textup{H.dim }\BB(\alpha_n) \log\left(\frac{1}{|c_F-\alpha_n|}\right)= 5 \log \frac{\sqrt{5}+1}{2}.
\end{equation}
\end{theorem}
We also show (Section \ref{feigenbaum}) that one can not do better than \eqref{modulus}, since  the quantity 
$\textup{H.dim }\BB(t) \log\left(\frac{1}{|c_F-t|}\right)$
does not admit a limit for $t \to c_F^-$.

Note that an analogous estimate for the modulus of continuity holds for the topological entropy of real quadratic polynomials; 
in that case, the entropy function is also not H\"older continuous at the Feigenbaum point, even if it is H\"older continuous everywhere else (see \cite{Ti3}). 
We conjecture that the same property holds for $\eta(t)$, namely that it is locally H\"older on $[0, 1] \setminus \{ c_F \}$.

\medskip 
\noindent \textbf{The recurrence spectrum and other relations.} 
Finally, as we mentioned, the combinatorial structure of our sets is related to Sturmian sequences. In particular, as a consequence of Theorem \ref{main} we also provide an answer to a question raised in \cite{Ca} (see Section \ref{Cassaigne}):

\begin{theorem} \label{T:Cass}
Let $\mathcal{R}$ be the recurrence spectrum for Sturmian sequences. Then, for each positive integer $N$, 
the Hausdorff dimension of $\mathcal{R} \cap [N+2, N+3]$ equals the Hausdorff dimension of $\BB_{N}$.
\end{theorem}

The computation of Hausdorff dimension for the sections of the recurrence spectrum is thus reduced to 
estimating the dimension of $\BB_N$. Since the latter is a regular Cantor set, there are 
very efficient algorithms to compute its dimension (\cite{JP}, \cite{He1}), and an 
asymptotic estimate is available as well (\cite{He3}, \cite{DFSU}). 

\medskip

Our definition of $\EE$ is somewhat reminiscent of the definitions of the Markoff and Lagrange spectra, whose local 
Hausdorff dimension has been studied in \cite{Mo}. 
Indeed, all these objects are defined by taking an infimum over orbits of the Gauss map (see \cite{Mo}, Section 1); however, it turns out
these sets are in fact different both topologically and metrically and a characterization as explicit as ours does not seem easily achievable 
for the Markov and Lagrange spectra.  


\medskip
Let us finally recall that the set $\EE$ is also connected to the set $\mathcal{U}$ of univoque numbers introduced in \cite{EHJ}, namely the values $q\in(1,2)$ for which $1$ admits a unique expansion in base $q$. More precisely there is an explicit order reversing homeomorphism between $\mathcal{U}$  and the  points of $\EE$  which are not periodic for $G$ (see  \cite{BCIT} \textsection 5.2); thus the value $c_F$, which is the largest non-isolated point of $\EE$ corresponds to the Komornik-Loreti constant $q_{KL}$ (see \cite{KL}), which is the smallest value in $\mathcal{U}$. Moreover, given a fixed base $q\in (1,2)$ one can consider the set $\mathcal{U}_q$ of numbers with a unique $q$-expansion, and this family of sets corresponds under the dictionary of \cite{BCIT} to the family $\BB(t)$; in fact, after the first draft of the present paper appeared, many combinatorial, topological or metric issues we have examined for $\BB(t)$ have been independently studied for $\mathcal{U}_q$ as well.

For instance, the continuity of the map $q\mapsto \textup{H.dim } \mathcal{U}_q$ has been addressed in the recent paper \cite{AK1}.
Moreover, entropy plateaus for $q$-expansions are described in \cite{ABD} and renormalization operators are used in \cite{AK2} and \cite{A}. Further, an analogue to Theorem \ref{main}(2) is proven for univoque numbers in \cite{KKLL}.

Sets with the same features as $\EE$ occur also in bifurcation problems for piecewise affine maps: we just mention the recent work \cite{DK}, but also \cite{BCMP}, where one can observe the very same structure of $\EE$ in every tuning window (see  \cite{BCMP}, Remark 4.29).
 
\medskip
\noindent \textbf{Acknowledgements.} 
We thank the anonymous referee for their detailed comments which improved the exposition.
G.T. is partially supported by NSERC and the Alfred P. Sloan Foundation. C.C. is partially supported by the research project PRIN 2017S35EHN\_004 ``Regular and stochastic behaviour in dynamical systems" of the Italian Ministry of Education and Research and by the GNAMPA group of the ``Istituto Nazionale di Alta Matematica" (INdAM).

\section{Preliminaries} 

\subsection{String action, cylinders, and regular Cantor sets} \label{notation}

The continued fraction expansion of a number $x$
will be denoted by $x = [0; a_1, a_2, \dots]$, or by $x = [0; S]$ where $S =  (a_1, a_2, \dots)$ is a (finite or infinite) 
string of positive integers.
Let us note that irrational numbers have infinite continued fraction expansions, and every rational number $r \in \mathbb{Q} \cap (0,1)$ has
exactly two finite expansions, one of even length and one of odd length; we will 
usually denote by $S_0$ the string of even length, and by $S_1$ the string of odd length ($r = [0; S_0]= [0; S_1]$).
Moreover, the expression $x = [0; \overline{S}]$ will denote the quadratic irrational whose continued fraction expansion
is periodic, given by $S$ repeated infinitely many times.
The length of the string $S$ will be denoted by $|S|$.

If $S$ is a finite string of positive integers we will denote by $q(S)$ the denominator 
of the rational number given whose c.f. expansion is $S$, i.e. $\frac{p(S)}{q(S)} = [0; S]$ with $(p(S), q(S)) = 1$, $q(S) > 0$. 

Let us also recall the well known estimate 
\begin{equation}\label{eq:supermult}
q(S)q(T) \leq q(ST)\leq 2q(S)q(T).
\end{equation}

Moreover, we can define an action of the semigroup of finite strings (with the operation of concatenation)
on the unit interval. Indeed, for each $S$, $S \cdot x$ is the number obtained by appending 
the string $S$ at the beginning of the continued fraction expansion of $x$. We shall also use the notation $f_S(x) := S \cdot x$.
If $S=(a_1,...,a_n)$ we can also write, by identifying matrices with M\"obius transformations, 

$$ S \cdot x := \matr{0}{ 1}{ 1}{ a_1} \matr{0}{ 1}{ 1}{ a_2}...  \matr{0}{
  1}{ 1}{ a_n} \cdot x.$$ 
It is easy to check that concatenation of
strings corresponds to composition  
$ (ST)\cdot x = S
\cdot(T\cdot x)$; moreover, the map $f_S$ is increasing
if $|S|$ is even, decreasing if it is odd.
The image of $f_S$ is a \emph{cylinder set} 
$$I(S) := \{ x = S \cdot y,\ y \in [0, 1] \},$$ 
which is a closed interval with endpoints
$[0; a_1, \dots, a_n]$ and $[0; a_1, \dots, a_n +1]$. 
The map $f_S$ is a contraction of the unit interval, and it is easy to see that 
\begin{equation} \label{contraction}
\frac{1}{2 q(S)^2} \leq |f'_S(x)| \leq \frac{1}{q(S)^2} \qquad \qquad \forall x \in [0,1],
\end{equation}
hence the length of $I(S)$ is bounded by
\begin{equation} \label{cylsize}
\frac{1}{2 q(S)^2} \leq |I(S)| \leq \frac{1}{q(S)^2}.
\end{equation}

Given two strings of positive integers $S = (a_1, \dots, a_n)$ and $T = (b_1, \dots, b_n)$ of equal length, let us define the 
\emph{alternate lexicographic order} as
$$S < T \ \textup{if } \exists k\leq n \textup{ s.t. } a_i = b_i \ \forall 1 \leq i \leq {k-1} \textup{ and } \left\{ \begin{array}{ll} a_n < b_n & \textup{if }n\textup{ even} \\
                                                                                                                       a_n > b_n & \textup{if }n\textup{ odd.} 
                                                                                                                      \end{array} \right.$$
For equal length strings we shall also write $S\leq T$ if either $S<T$ or $S=T$. 
The importance of such order lies in the fact that (given two strings of equal length) $S < T \textup{ iff } [0; S] < [0; T]$. 
Moreover, in order to compare quadratic irrationals with periodic expansion, we shall repeatedly use the so called \emph{string lemma}
\begin{lemma}[\cite{CT}, Lemma 2.12]\label{stringlemma} 
For any pair of strings
$S$, $T$ of positive integers, 
\begin{equation} 
ST < TS \Leftrightarrow [0; \overline{S}] < [0; \overline{T}].
\end{equation}
\end{lemma}

The order $<$ is a total order on the strings of positive integers of fixed given length; to be able to compare strings of different lengths 
we will define the partial order 

$$S  < <   T \quad \textup{if }\exists i \leq \min\{|S|, |T|\} \textup{ s.t. }(S)_1^i < (T)_1^i$$
where $(S)_1^i = (a_1, \dots, a_i)$ denotes the truncation of $S$ to the first $i$ characters. 
Let us note that:

\begin{enumerate}
\item If $|S| = |T|$, then $S < T$ iff $S  < <   T$.
\item If $S, T, U$ are any strings, $S  < <   T \Rightarrow SU  < <   T, S  < <   TU$. 
\item If $S  < <   T$, then $S \cdot z < T \cdot w$ for any $z, w \in (0, 1)$.
\end{enumerate}

In general, the computation of the Hausdorff dimension of a compact subset of the interval is a non-trivial task: 
however, there is a family of sets whose dimension we can estimate easily, namely sets of numbers 
whose continued fraction expansion is given by concatenations of words of some finite alphabet:

\begin{definition} \label{RCS}
Given a finite set of strings $\mathcal{A}$, the \emph{regular Cantor set} defined by $\mathcal{A}$ is the set
$$K(\mathcal{A}) := \{ x = [0; W_1 W_2 \dots ] \ : \ W_i \in \mathcal{A} \ \forall i \geq 1\}.$$
\end{definition}

A regular Cantor set is a particular case of an {\em iterated function system} and its Hausdorff dimension can be estimated in a
standard way. Let us assume that the alphabet $\mathcal{A}$ is not redundant, in the sense
that for each $W_i$ and $W_j$ distinct elements of $\mathcal{A}$ we have $W_i\cdot K(\mathcal{A}) \cap W_j\cdot K(\mathcal{A})=\emptyset$.
Now, let $m_j := \inf_{x \in [0, 1]} |f_{W_j}'(x)|$ and $M_j := \sup_{x \in [0, 1]}  |f_{W_j}'(x)|$
be respectively the smallest and largest contraction factors of the
maps $f_{W_j}$.
Then we have the estimates (\cite[Proposition 9.6, 9.7]{Fa})
\begin{equation} \label{dimbounds}
s_1 \leq \textup{H.dim } K(\mathcal{A}) \leq s_2
\end{equation}
where $s_1$ is the root of $ \sum_j m_j^s=1$ and $s_2$ is the root of $ \sum_j M_j^s=1$.
Note that the upper bound also holds without the disjointness assumption.

\subsection{Generalized bounded type numbers and their bifurcation locus} \label{biflocus}

Let us start by proving some elementary facts, which will yield Proposition \ref{prima}
stated in the introduction, namely that the bifurcation locus of the family $\BB(t)$ 
is the \emph{exceptional set}

$$\EE := \{ x\in [0, 1] \ : \ G^k(x) \geq x \  \forall k \in \mathbb{N} \}.$$

\begin{lemma}\label{B}
The sets $\BB(t)$ have the following properties:
\begin{enumerate}
\item[(i)] $\BB(0)=[0,1]$;  $\BB(t)=\emptyset$ if $t>g = \frac{\sqrt{5}-1}{2}$, in fact
$t \mapsto \BB(t)$ is monotone decreasing;
\item[(ii)]
$\BB(t)$ is forward-invariant for the Gauss map $G$;
\item[(iii)]
$\BB(t)$ is closed and, if $t>0$, with no interior and of zero Lebesgue measure;
\item[(iv)]
the union $\bigcup_{t > 0} \BB (t)$ is the set of bounded type numbers;
\item[(v)]
$\bigcap_{t'<t}\BB(t')=\BB(t)$;
\item[(vi)]
$\EE=\{t \in [0,1] \ : \ t \in \BB(t)\}.$
\end{enumerate}
\end{lemma}

\begin{proof}
Points (i), (ii), (iv), (v), (vi) are immediate by definition. \newline
{\bf (iii)} 
Let us consider the Farey map $F : [0, 1] \rightarrow [0,1]$, defined by
$$F(x) := \left\{ \begin{array}{ll} 
                  \frac{x}{1-x} & \textup{if }0 \leq x \leq \frac{1}{2} \\
		  \frac{1-x}{x} & \textup{ if }\frac{1}{2} < x \leq 1.
                 \end{array} \right.$$

One can easily check that if $x:=[0;a_1,a_2,a_3,...]$ then $F(x)=[0;a_1-1,a_2,a_3,...]$ if $a_1>1$ (while 
  $F(x)=[0;a_2,a_3,...]$ in the case $x:=[0;1,a_2,a_3,...]$) and so it is clear that
 for each $x \in [0,1]$
$$\inf_{k \geq 1} G^k(x) = \inf_{k \geq 1} F^k(x).$$
Therefore one can write
\begin{equation}\label{eq:BBF}
\BB(t) = \{ x \in [0,1] \ : \ F^k(x) \geq t \ \forall k \in \NN \},
\end{equation}
which is closed by continuity of $F$. For $t >0$, $\BB(t)$ has no interior because 
it does not contain any rational number, and it has measure zero by ergodicity of 
the Gauss map.
\end{proof}

Since $\BB(t)$ is a non-empty compact set for $0 \leq t \leq g$, for such values of $t$ one can define the function
$$\ell(t) := \min \BB (t).$$
We shall list some elementary properties of $\ell$.

\begin{lemma}\label{ell}
The function $t \mapsto \ell(t)$ is monotone increasing and  
\begin{enumerate}
\item[(i)]
For any $0 \leq t \leq g$, $\ell(t) \in \EE$;
\item[(ii)]
$t \leq \ell(t) \ \forall t\in [0,g]$;
\item[(iii)]
$t=\ell(t) \iff t\in \EE$; 
\item[(iv)] $\ell$ is left-continuous: $\ell(t)=\sup_{t'<t}\ell(t')= \lim_{t'\to t^-} \ell(t')$;
\item[(v)]
if $(\alpha,\beta)$ is a connected component of $[0,g]\setminus \EE$ then
$$
\begin{array}{c}
\ell(t)=\beta\\
\BB(t)=\BB(\beta)
\end{array}
\ \ \ \ \forall t \in (\alpha, \beta].
$$
\end{enumerate}

\end{lemma}

\begin{proof}
(i): Since $\ell(t) \in \BB(t)$, $G^n(\ell(t)) \geq \ell(t)$, hence $\ell(t) \in \EE$.
(ii): $ x \in \BB(t) \Rightarrow x \geq t$, hence $\ell(t) \geq t$.
(iii) is a consequence of Lemma  \ref{B}-(vi): 
$$ t=\ell(t) \  \iff \ t \in \BB(t) \ \iff \ t \in \EE.$$
(iv) follows from Lemma \ref{B}-(v).
(v): let us pick $t$ such that $\alpha<t<\beta$. Since 
$(\alpha, \beta)$ is a connected component of $[0,g]\setminus
  \EE$ we have $\beta \in \EE$ and so $\beta \in \BB( \beta)\subset\BB(
  t)$,  and
$$\beta\geq \min \BB(t)=\ell(t).$$
On the other hand, since $(\alpha,\beta)\cap \EE=\emptyset$ 
and $\ell(t)\in \EE\cap [t,1]$ 
it follows that
$$\ell(t)\geq \beta.$$
We have thus proved that $\ell(t)=\beta$.
Now, from (ii) and monotonicity, $\BB(\ell(t)) \subseteq
 \BB(t)$. Moreover, if $x \in \BB(t)$, by $G$-invariance $G^n(x) \in
\BB(t)$, hence $G^n(x) \geq \ell(t)$ and $x \in \BB(\ell(t))$, hence
$\BB(t) = \BB(\ell(t)) = \BB(\beta)$.
\end{proof}
From Lemma \ref{ell} it also follows that
$$\ell(t) = \min(\EE \cap [t, 1]).$$

\textit{Proof of Proposition \ref{prima}}. By Lemma \ref{ell}-(v), 
the function $t \mapsto \BB(t)$ is locally constant outside $\EE$.
On the other hand, if $t \in \EE$, then $t \in \BB(t)$ by definition, but 
$t \notin \BB(t')$ for any $t'>t$, so $t$ must belong to the bifurcation set. \qed
 
\section{A complementary point of view} \label{complementary}

If one wants to figure out the structure of a closed, nowhere dense set $C
\subset [0,1]$ (even just in order to plot a picture), one usually
tries to describe its complement: indeed, while $C$ is typically
uncountable and totally disconnected, its complement $A:=[0,1] \setminus C$ is a
countable union of open connected components (intervals): $$ A= \bigcup_{j\in \NN}I_j.$$ 
We recall the following simple fact (\cite{CT}, Lemma 2.1):

\textit{If $(a,b) \subset [0,1]$ there is a {\bf unique} rational number $p/q \in (a,b)\cap
\QQ$ with minimal denominator, i.e. such that  $q\leq q'$ for all other $p'/q' \in (a,b)\cap \QQ$.
We shall call such a ``minimal" rational the {\em pseudocenter} of the interval  
$(a,b)$.}

The above fact provides a canonical way of labelling the
connected components of $[0,1]\setminus C$; in fact if $C$ is a closed, possibly fractal
set of irrational numbers which is defined imposing a certain
condition on the continued fraction of its elements, it often happens
that the pseudocenters of the connected components of $[0,1]\setminus C$ are
rational values which satisfy, at a finite stage, the same property
which defines $C$: this is known to be true for diophantine-type
conditions arising in the solution of cohomological equations \cite{MS}, 
for the sublevels of the Brjuno function \cite{CM} and -indeed- 
also for $\EE$ and $\BB (t)$, as we shall see in a moment.

\subsection{Extremal rational numbers and the exceptional set} \label{extrat}

The bifurcation set $\EE$ is a totally disconnected, nowhere dense compact set. 
An explicit description of the connected components of its complement was first given in \cite{CT}.  
In this section we shall describe this construction, and we refer to \textsection 3.1 of \cite{BCIT} 
for more details.

Let us recall that any $r\in (0,1) \cap \QQ$ admits two continued fraction expansions, one of even and one of odd length (for instance $3/10=[0; 3, 3]=[0;3,2,1]$). Using the string notation introduced in \textsection \ref{notation} we can write $r = [0; S_0] = [0; S_1]$ with $|S_0|$ even and $|S_1|$ odd;  the two strings $S_0, S_1$ defining the same rational value $r\in(0,1)$ are called {\it conjugate strings}, and one has $S_1 < <  S_0$. 

For any rational value $0 < r < 1$ we then define the {\em quadratic interval} of pseudocenter $r$ as $I_r := (\alpha^-, \alpha^+)$ where $\alpha^- = [0; \overline{S_1}]$, 
$\alpha^+ = [0; \overline{S_0}]$; we also define $I_r$ for $r=1$ setting  $I_1 := (g, 1]$. It turns out that any such quadratic interval is contained in some maximal one (\cite{CT}, Lemma 2.5), and {\em maximal quadratic intervals} are the connected components of $[0,1]\setminus \EE$.

Let us recall the following characterization of the pseudocenters of maximal quadratic intervals. 

\begin{proposition}[\cite{CT}, Proposition 4.5]\label{P45}
Let $r\in(0,1)$ be a rational number. Then $r=[0;S]$ is the pseudocenter of a maximal quadratic interval iff for every nontrivial factorization $S=XY$ either $XY<YX$ or $X=Y$ and $|X|$ is odd.
Moreover, if one of the conditions above is true, then for every factorization $S=XY$ it happens that $r=[0;S]<[0;Y]$.
\end{proposition}

Let us now give a couple of definitions which can be used to provide alternative characterizations of the pseudocenters.

\begin{definition}
A rational number $r \in \mathbb{Q} \cap (0,1]$ is called \emph{extremal} if there exists $n \geq 1$ such that $G^n(r) = 0$ and 
$G^k(r) > r$ for any $1 \leq k \leq n-1$.
\end{definition}

The continued fraction expansions of extremal rational numbers have a particular structure, namely they are extremal in the following sense.

\begin{definition} \label{D:ext-string}
A finite string of positive integers $S$ is called \emph{extremal} if 
$$XY < YX$$ 
for every splitting $S = XY$ where $X$, $Y$ are nonempty strings.
\end{definition}

The following lemma puts together all the above results and definitions.

\begin{lemma}\label{P45CT1}
Let $r \in \mathbb{Q} \cap (0,1]$ be a rational number. Then the following conditions are equivalent:
\begin{itemize}
\item[(i)] $r$ is extremal;
\item[(ii)] the interval $I_r$  is a maximal quadratic interval;
\item[(iii)] if  $r = [0; S_1]$ is the continued fraction expansion of odd length of $r$, then $S_1$ is extremal.
\end{itemize}
Moreover, if any of these conditions holds and $r = [0; S_0]$ is the continued fraction expansion of even length of $r$, then 
for all splittings $S_0 = XY$ one has$$XY \leq YX.$$
\end{lemma}

\begin{proof}
If $r=1$ there is hardly anything to prove (note that we consider $r=1$ extremal as well, and $S=(1)$ is extremal for free). For $r\in (0,1)$ it is immediate to check that Proposition \ref{P45} gives the implications $(ii) \iff (iii) \Rightarrow (i)$, hence we only need to prove $(i) \Rightarrow (ii)$. 


To do so, 
suppose that $r \in \mathbb{Q}$ is such that $G^k(r) > r$ for $1 \leq k \leq n-1$ and $G^n(r) = 0$. 
Let $(\alpha^-, \alpha^+)$ be the complementary interval of $\EE$ where $r$ lies, and suppose the pseudocenter of such a component is 
$s := [0; Z_0] = [0; Z_1]$, so that $\alpha^- = [0; \overline{Z_1}]$ and
$\alpha^+ = [0; \overline{Z_0}]$. 
Then either $r = Z_0 \cdot r'$ or $r = Z_1 \cdot r'$ for some $r' \geq 0$.  
In the first case, if $k = |Z_0|$ then 
$G^k([s, \alpha^+)) = [0, \alpha^+)$.
Hence, since $G^k$ is expanding, $r' = G^k(r) < r$. 
In the second case, if $k = |Z_1|$ then $G^k((\alpha^-, s]) = [0, \alpha^-)$ and $G^k$ is orientation reversing, hence $r' = G^k(r) < r$.
In both cases, the fact that $r$ is extremal implies $r' = 0$, so $r = s$ is the pseudocenter of $(\alpha^-, \alpha^+)$, hence $I_r = I_s$ is maximal.

The last claim follows immediately from Proposition \ref{P45}.
\end{proof}

As a consequence of this discussion, there is a bijection between extremal rational numbers and maximal quadratic intervals; therefore the set

\begin{equation}\label{eq:QE}
\mathbb{Q}_E = \{ r \in \mathbb{Q} \cap (0, 1], \ r \textup{ extremal} \} 
\end{equation}
allows us to label all maximal quadratic intervals, which are the connected components of $[0,1]\setminus \EE$. Thus we have shown that the complement of $\EE$ can be expressed as a disjoint union
\begin{equation} \label{complementE}
[0, 1] \setminus \EE := \bigcup_{r \in \mathbb{Q}_E} I_r.
\end{equation}


It is also possible to describe the topology of $\EE$ in detail, and more precisely identify its isolated points. 
Indeed, if $r = [0; S_1]$ is extremal, then $r_1 = [0; S_1 S_0]$ is also extremal (\cite{CT}, Proposition 3.9), and 
$\EE \cap [r_1, r]$ consists of the single,  isolated point $\alpha_1 = [0; \overline{S_1}]$. All isolated points are obtained in this way, and come grouped in \emph{period-doubling cascades} since we can reiterate this scheme as we explain below.

Let us denote by $Z'$ the conjugate string of $Z$ (so that $[0; Z'] = [0; Z]$). 

\begin{proposition}[\cite{BCIT}] \label{P:pd}
Let $r = [0; S_1]$ be an extremal rational number, with $S_1$ its continued fraction of odd length. Define the sequence
$$
\begin{cases}
Z_1=S_1\\
Z_{n+1}=Z_n Z'_n, \ \ (n\geq 1).
\end{cases}
$$
Then all the values $\alpha_n:=[0;\overline{Z_n}]$ belong to $\EE$ and form a decreasing sequence converging to a point $\alpha_\infty \in \EE$  which has an aperiodic continued fraction expansion. In fact, if $\alpha^+ = [0; \overline{S_0}]$, then $\EE \cap (\alpha_\infty, \alpha^+)$ is precisely the discrete set $\{\alpha_n, \ n\geq 1\}$, which we call the \emph{period-doubling cascade} associated to $r$.

Finally, every isolated point of $\mathcal{E}$ belongs to some period-doubling cascade.
\end{proposition}
 
For instance, by starting with $r = \frac{1}{2} = [0; 2] = [0; 1, 1]$, one gets the sequence 
$$[0; \overline{2}], [0; \overline{2, 1, 1}], [0; \overline{2, 1, 1, 2, 2}], [0; \overline{2, 1, 1, 2, 2, 2, 1, 1, 2, 1, 1}], \dots $$
which tends to the point $c_F = 0.3867499707...$. Notice that in the dictionary of \cite{BCIT}, this point corresponds to 
the Feigenbaum parameter for unimodal maps (hence the notation $c_F$). 
It is also easy to describe the sequence of strings $Z_n$ in terms of the conjugate string $S_0$ and $S_1$:
$$  {\tiny S_1, \ \  S_1S_0, \ \ S_1S_0S_1S_1, \ \ S_1S_0S_1S_1S_1S_0S_1S_0,  \ \ \dots}$$
Note that the sequence can be obtained by successive iterations of the substitution operator 
$$\begin{cases}
S_1 \mapsto S_1S_0\\
S_0 \mapsto S_1S_1.\\
\end{cases}
$$
Later on we will use this pattern to find the modulus of continuity of $\eta(t) =  \textup{H.dim }\BB (t)$ near $c_F$; in fact in Section \ref{tuning} we will also see how period-doubling fits in the more general framework of tuning operators (cf. Corollary \ref{c:pd}).

\subsection{Anatomy of generalized bounded type numbers}\label{anatomy}

In this section, we will provide an analoguous description for the set $\BB(t)$, 
by producing the set of pseudocenters of all connected components of its complement. We need two definitions: 

\begin{definition}
Let $t\in (0,g]$ be a fixed value, $r:=[0;S_0]=[0;S_1] \in \QQ\cap (0,1)$ and let us define
\begin{equation}
V_r:=(S_1 \cdot \ell(t), S_0 \cdot \ell(t)).
\end{equation} 
We also define $V_0=[0, \ell(t))$ and $V_1:=(1/(1+ \ell(t)), 1]$. We shall call $V_r$ the $t$-gap generated by $r\in \QQ \cap [0,1]$.
\end{definition}

\begin{definition}
Let $t\in (0,1)$ and $r\in \QQ \cap [0,1]$. We say that  $m\in \NN$ is the \emph{running time} of $r$ if
\begin{equation}\label{eq:rtime} 
G^m(r)=0 \ \ \mbox{ but } G^k(r) >0 \ \ \forall 0\leq k \leq m-1. 
\end{equation}
We say that $r$ is a \emph{$t$-label} 
if the $G$-orbit of $r$ never enters the interval $(0,\ell(t))$;
this means that if $m$ is the running time of $r$ 
\begin{equation}\label{eq:tlabel} 
G^m(r)=0 \ \ \mbox{ but } G^k(r) > \ell(t) \ \ \forall 0\leq k \leq m-1. 
\end{equation}
We shall denote by $\QQ_t$ the set of all $t$-labels.
\end{definition}

We will prove that $t$-gaps provide a complete description of $\BB(t)$ by determining the connected components of its complement: 

\begin{proposition} \label{anatomyprop}
Let $t \in (0, g]$. Then the connected components of the complement of $\BB(t)$ are precisely
the $t$-gaps generated by the $t$-labels, hence
$$\BB(t) = [0, 1] \setminus \bigcup_{r \in \mathbb{Q}_t} V_r.$$
\end{proposition}

\begin{corollary} \label{isolated}
$\BB(t)$ contains isolated points if and only if $\beta := \ell(t)$ is isolated in $\EE$.
\end{corollary}

\begin{proof}
Let us set $\beta:=\ell(t)$. Let us recall that $\beta$ is
isolated in $\EE$ iff $\beta=[0;\overline{S}]$ with minimal period $S$ of odd length (\cite{BCIT}, Lemma 4.5).

Now, we shall prove that if $\beta$ is isolated in $\EE$, then it is isolated in $\BB(t) = \BB(\beta)$. 
Indeed, if $\beta=[0;\overline{S}]$ has odd minimal period, then $S\cdot 0>\beta$. Moreover if
 $x\in (\beta, S\cdot 0)$ then $x = S \cdot y > S \cdot \beta = \beta$ for some $y$, hence, since $|S|$ 
is odd, $G^{|S|}(x) = y < \beta$. This means that $x$ does not belong to $\BB(t) = \BB(\beta)$, 
hence $\beta$ is isolated in $\BB(\beta)$.

Conversely, if $x$ is an isolated point of $\BB(t)$ then by Proposition \ref{anatomyprop} $x$ is the common endpoint of two adjacent $t$-gaps; therefore we can write 
$x=S\cdot \beta=T\cdot \beta$ where $S,T$ are two strings, one of even
and one of odd length. If -say- $S$ is the longest then $S=TR$ with
$|R|$ odd, therefore $R\cdot \beta =\beta$, which means that $\beta$
is periodic with minimal period of odd length.
\end{proof}

\begin{corollary}
The points of discontinuity of the set-valued function
$$t \mapsto \BB(t)$$ in the Hausdorff topology are precisely the
isolated points of $\EE$ (corresponding to period-doubling bifurcation
parameters).
\end{corollary}

\begin{proof}

 If $A$ is a set, we shall set $N_\epsilon(A):=\{x: d(x,A)
<\epsilon\}$; let us also recall that if $A,B$ are two nonempty closed
sets, then the Hausdorff distance $d_H(A,B)$ is smaller than
$\epsilon$ iff $B\subset N_\epsilon (A)$ and $A\subset N_\epsilon
(B)$; let us point out that, since the sets we are interested in are
monotone, we will always have to check just one inclusion.

It is easy to prove that $t\mapsto \BB(t)$ is everywhere left-continuous 
with respect to Hausdorff distance: this amounts to checking
that, if $t\in (0,1)$ and $t_n \nearrow t$, then for every fixed
$\epsilon>0$ there exists $\bar{n}$ such that $\BB(t_{\bar{n}})\subset
N_\epsilon(\BB(t))$. This is a straightforward consequence of
$\BB(t)=\cap_{t'<t} \BB(t')$.

On the other hand, because of Corollary \ref{isolated} the map cannot be right-continuous 
at isolated points of $\EE$, so we just have to prove that
it is right-continuous at any $\beta$ a non-isolated point of $\EE$; 
such points are always accumulated on the right by points of $\EE$, hence
$$\exists (\beta_n)\subset \EE \ :
\ \beta_n \searrow \beta \ \mbox{ as } n \to +\infty.$$

To prove right-continuity at any such $\beta \in \EE$ it will be
enough to check that for all $\epsilon >0$ there exists $\bar{n}$ such
that $\BB(\beta)\subset N_\epsilon(\BB(\beta_{\bar{n}}))$.  Since the
set $\{ S\cdot \beta \ : \ [0;S]\in \QQ_\beta\}$ is dense in
$\BB(\beta)$ we can find a finite subset $\mathcal{S}\subset \{S \ :
\ [0;S] \in \QQ_\beta\}$ such that $\BB(\beta)\subset
N_{\epsilon/2}(E)$ with $E:=\{ S\cdot \beta \ : \ \beta\in
\mathcal{S}\}$. Therefore if we prove that $$E\subset
N_{\epsilon/2}(\BB(\beta_{\bar{n}})) \ \ \ \ (*)$$ we are done.  For
$S\in \mathcal{S}$ we denote $\underline{r}(S)$ the minimum positive
element of the orbit $G^k(S\cdot 0)$; since $\mathcal{S}\subset
\QQ_\beta$ we have that $\underline{r}(S)>\beta$, therefore
$\underline{r}:=\min_{S\in \mathcal{S}} \underline{r}(S)>\beta$ as
well. Now let us choose $\bar{n}$ such that
$$ (a) \ \beta_{\bar{n}} \in (\beta, \underline{r}), \ \ \ \ (b)
\ |\beta_{\bar{n}} - \beta|<\epsilon/2;$$ by $(a)$ we have that
$\mathcal{S}\subset\QQ_{\beta_{\bar{n}}}$ and by Proposition \ref{anatomyprop} $S\cdot \beta_{\bar{n}} \in
\BB(\beta_{\bar{n}}) $ for all $S\in \mathcal{S}$. On the other hand $(b)$ implies that 
$$ |S\cdot \beta -S\cdot \beta_{\bar{n}}|\leq|\beta-\beta_{\bar{n}}|<\epsilon/2 \ \  \ \forall S\in \mathcal{S},$$
and so condition (*) above holds.

\end{proof}

In the following, we will denote $\beta := \ell(t)$. 
In order to prove Proposition \ref{anatomyprop}, let us first observe some basic properties of $t$-labels (recall that $\BB(t) = \BB(\beta)$):
\begin{itemize}
\item If $V_r$ is a $t$-gap, then $V_r \cap \BB(t)=\emptyset$;
\item if $V_r$ is the $t$-gap generated by $r$ and $G^k(r) \neq 0$, then $G^k(V_r)$ is the $t$-gap generated by $G^k(r)$.
\item $[0,1]\setminus \BB(t)= \bigcup_{r\in [0,1]\cap \QQ} V_r$.
\end{itemize}

In the formula above, in the union lots of overlappings occur;
nonetheless we shall see that $\BB(t)$ can be described as the
disjoint union of a subfamily of $t$-gaps.

\begin{lemma}\label{one}
Let $\beta \in \EE$ be a fixed value, $r:=[0;S] \in \QQ\cap [0,1]$, $r>\beta$. Then $S\cdot \beta \geq \beta$.
\end{lemma}

\begin{proof}
If $|S|$ is even the claim is trivial:
$$S\cdot \beta \geq S\cdot 0 = r >\beta.$$ If $|S|$ is odd, set
$\alpha^-:=[0; \overline{S}]$; $\alpha^-$ is the left endpoint of the
quadratic interval $I_r$ generated by $r$, so $\alpha^-<r$ but also
$\beta \leq \alpha^-$ (because no point of $\EE$ can belong to $I_r$).
Hence $S \cdot \beta \geq S \cdot \alpha^- = \alpha^- \geq \beta$.
\end{proof}

\begin{lemma}\label{two}
Let $t \in (0, 1)$, and $r:=[0;S_1]=[0;S_0] \in \QQ_t$ be a $t$-label, $\beta := \ell(t)$.
 Then the $t$-gap generated by $r$, namely the open interval
 $V_r:=(S_1\cdot \beta,S_0\cdot \beta)$, is  a connected component of $[0,1]\setminus \BB(t)$.
\end{lemma}

\begin{proof}
If $r=0$ or $r=1$ there is no problem. Otherwise $r=[0;S_1]=[0;S_0]$ and
$V_r \subset [0,1] \setminus \BB(t)$; to prove our claim it is sufficient to check that
$\partial V_r \subseteq \BB(t)$, i.e. both $\beta^+:=S_0 \cdot \beta$ and $\beta^-:=S_1 \cdot \beta$ belong to the set $ \BB(\beta)$.
By assumption we have that $G^k(r)>\beta$ for  $0\leq k \leq m-1$; moreover $\partial V_{G^k(r)}=\{G^k(\beta^\pm)\}$ and hence
 $G^k(\beta^\pm) \geq \beta$ for $0\leq k\leq m-1$ (by Lemma \ref{one}).
On the other hand $G^m(\beta^\pm)=\beta$ and hence $\beta^\pm \in \BB(\beta)$.
\end{proof}

\begin{lemma}\label{three}
Let $t\in (0,1)$, let $V:=(\xi^-, \xi^+)$ be a  connected component of $[0,1]\setminus \BB(t)$, and let $r\in V\cap \QQ$.
The following conditions are equivalent:
\begin{enumerate}
\item[(i)] $r$ is the pseudocenter of $V$;
\item[(ii)] $r$ is a $t$-label.
\end{enumerate} 

\end{lemma}

\begin{proof}
$(i)\Rightarrow (ii)$ We argue by contradiction. Assume that $r\in \QQ$ is
the pseudocenter of $V$ but there exists $k_0$ such that $r_1:=G^{k_0}(r)
\in (0,\ell(t)).$ In this case $r=[0; S_0 S_1]$ and $r_1=[0;S_1]$ for certain strings $S_0, S_1$; moreover
setting $r_0:=[0; S_0]$ we see that the interval $[0,r_1] \subset
[0,\ell(t))$, hence $S_0 \cdot [0,r_1] \subset [0, 1] \setminus \BB(\ell(t)) = [0,1]\setminus \BB(t)$, and so $S_0 \cdot [0,r_1]$
    must also be contained in $V$, by connectedness. This is a
    contradiction: if we denote as den$(r)$ the denominator of $r$ in reduced form, then $r_0=S_0\cdot 0 \in V$ and den$(r_0) < $ den$(r)$
    imply that $r$ cannot be the pseudocenter of $V$.

$(ii)\Rightarrow (i)$ Let $\beta:= \ell(t)\in \EE$, so that
$\BB(t)=\BB(\beta)$. By Lemma \ref{two}, if $r:=[0;S_1]=[0;S_0] \in V$
satisfies (ii) then $V=(S_1\cdot \beta, S_0 \cdot \beta)$ and hence
$r$ is the pseudocenter of $V$.
\end{proof}

\textit{Proof of Proposition \ref{anatomyprop}}. 
If $r$ is a $t$-label, then the $t$-gap $V_r$ is a connected component of $[0,1] \setminus \BB(t)$ by Lemma \ref{two}. 
On the other hand, if $V$ is a connected component, then its pseudocenter $r$ is a $t$-label by Lemma \ref{three}, hence 
$V = V_r$ by Lemma \ref{two}. \qed
\medskip

\begin{remark}
In order to compute $\ell (t)$ given $t$, one can proceed as follows: if $t\in \EE$ then $t=\ell(t)$; if not, then $t$ belongs to some maximal interval $(\alpha^-, \alpha^+)$ and we can find  the smallest integer $m$ such that  $G^m(t)<t$; thus $t=S\cdot y$ with $y = G^m(t)$ and the  string $S$, which consists of the first $m$ partial quotients of $t$, is an expansion of the pseudocenter of $(\alpha^-,\alpha^+)$ so that $\ell(t)=\alpha^+=[0;\overline{S_0}]$ with $S_0=S$ if $m$ is even, $S_0=S'$ if $m$ is odd.
\end{remark}

\section{Tuning} \label{tuning}

The goal of this section is to define renormalization operators for numbers of bounded type, and identify their action on parameter space.

Let $r \in \mathbb{Q} \cap (0,1)$, and let $r = [0; S_0] = [0; S_1]$ respectively the even and odd continued fraction expansion of $r$.
The \emph{tuning window} associated to $r$ is the interval $W_r := [\omega, \alpha)$ with endpoints
$$\omega := [0; S_1 \overline{S_0}],$$
$$\alpha := [0; \overline{S_0}].$$

We can define the \emph{tuning operator} $\tau_r := [0, 1] \setminus \mathbb{Q} \mapsto [0, 1]$ as
$$[0; a_1, a_2, \dots] \mapsto [0; S_1 S_0^{a_1-1}S_1S_0^{a_2-1}\dots]$$
and $\tau_r(0) = \omega$. It is easy to check that the map $x \mapsto \tau_r(x)$ is strictly increasing, and its image is 
contained in $W_r$ (indeed, $\tau_r(\mathcal{E}) = \mathcal{E} \cap W_r$). Moreover, tuning behaves well with respect to the string action:
$$\tau(S\cdot x) = \tau(S) \cdot \tau(x)$$
where $S = (a_1, \dots, a_n)$ is a finite string and $\tau(S) := S_1 S_0^{a_1-1} \dots S_1S_0^{a_n -1}$.

The point $\alpha := [0; \overline{S_0}]$ will be called the \emph{root} of the tuning window. 
A tuning window will be called \emph{maximal} if it is not contained in any other tuning window.
An extremal rational number $r$ will be called \emph{primitive} if $W_r$ is a maximal tuning window.

Let us remark that the period-doubling cascades of Section \ref{extrat} are generated by iteration of 
the tuning map $\tau_{1/2}$. 
For instance, if $g = [0; \overline{1}]$, then $\tau_{1/2}(g) = [0; \overline{2}]$, $\tau^2_{1/2}(g) = [0; \overline{2, 1, 1}]$ and so on, hence
$\tau_{1/2}^n(g)$ is the cascade which tends to $c_F$, and $c_F$ is a fixed point of $\tau_{1/2}$. It is then easy to check that all other period-doubling cascades are tuned images of this one:

\begin{corollary}\label{c:pd} 
For any  extremal rational number $r$, the cascade generated by $r$ is just the sequence $\{\tau_r \tau_{1/2}^n(g), \ n \geq 0 \}$.
\end{corollary}

In this section, we will determine the structure of the tuning windows. Let us start with a preliminary result.

\begin{lemma}\label{prosuf}
Let $r=[0;S_0]=[0;S_1]\in \QQ_E$ and  $y:= [0;  B, S_*, \dots]$ with $B$ a proper suffix of either $S_0$ or $S_1$,
and $S_*$ equal to either $S_0$ or $S_1$. Then $y> [0;S_1]$.
\end{lemma}
\begin{proof}
If $B=(1)$ then there is hardly anything to prove (the first digit of $S_1$ is strictly greater than $1$).
If not, then one of the following is true:
\begin{enumerate}
\item $S_0=AB$ and $A$ is a prefix of $S_1$ as well; 
\item $S_1=AB$ and $A$ is a prefix of $S_0$ as well. 
\end{enumerate}
By Lemma \ref{P45CT1}, in the first case we get that $BA\geq AB=S_0 >> S_1$, while in the latter we get that $BA >> AB=S_1$; so in both cases
$BA >> S_1$ and the claim follows.
\end{proof}

Let us now observe that the structure of $\BB(\omega)$ inside the tuning window is particularly nice:

\begin{lemma} \label{concat}
Let $r \in \mathbb{Q}\cap(0,1)$ be extremal, $W_r = [\omega, \alpha)$. Then 
$$\BB(\omega) \cap [\omega, \alpha] = K(\Sigma)$$
where $K(\Sigma)$ is the regular Cantor set on the alphabet $\Sigma = \{ S_0, S_1 \}$.
This means that $x$ belongs to $B(\omega) \cap [\omega, \alpha]$ if and only if its continued fraction
is an infinite concatenation of the strings $S_0$, $S_1$.
\end{lemma}

\begin{proof}
If $x \in \BB(\omega) \cap [\omega, \alpha]$, either $x = S_0 \cdot y$ for some $y$, or $x = S_1 \cdot y$. In the first case, 
from $x \leq [0; \overline{S_0}]$ and the fact $|S_0|$ is even it follows that $y \leq [0; \overline{S_0}]$. In the second case, from $x \geq [0; S_1, \overline{S_0}]$ and the fact $|S_1|$ is 
odd it also follows that $y \leq [0; \overline{S_0}]$. Moreover, since $x \in \mathcal{B}(\omega)$, the point
$y = G^{k}(x)$ belongs to $\mathcal{B}(\omega)$ and also $y \geq \omega$, hence $y \in \mathcal{B}(\omega) \cap [\omega, \alpha]$ and the inclusion is proved by induction. 

Conversely, if $x \in K(\Sigma)$, then either  $G^k(x)\in  K(\Sigma)$ as well, and hence $G^k(x)\geq \omega=\min  K(\Sigma)$, or $G^k(x) = [0;  B, S_*, \dots]$ with $B$ a proper suffix of either $S_0$ or $S_1$, 
and $S_*$ equal to either $S_0$ or $S_1$; in this latter case $G^k(x)>\alpha_1 = [0; \overline{S_1}] >\omega$ is a consequence of Lemma \ref{prosuf}.
\end{proof}

Note that the previous lemma suggests a possible strategy to prove Theorem \ref{main}(3), by comparing the dimensions of $\mathcal{B}(\alpha)$ and 
$K(\Sigma)$. However, such comparison is not completely trivial as $\mathcal{B}(\alpha)$ is not in general a regular Cantor set. The complete proof of Theorem \ref{main}(3) will be given in Section \ref{S:local}. 

\begin{lemma} \label{induced_order}
Let $r \in \mathbb{Q}\cap (0,1)$ be extremal, and $x, y \in [0, 1] \setminus \mathbb{Q}$. Then
$$G^k(x) \geq y \ \ \forall k \geq 0$$
if and only if 
$$G^k(\tau_r(x)) \geq \tau_r(y) \ \ \forall k \geq 0.$$
\end{lemma}
 
\begin{proof}
Let us set $r=[0;S_0]=[0;S_1]$ and $x=[0;a_1,a_2,a_3,...]$.
Since $\tau_r$ is increasing, $G^k(x) \geq y$ if and only if
$\tau_r(G^k(x)) \geq \tau_r(y)$ if and only if $G^{N_k}(\tau_r(x))
\geq \tau_r(y)$ for $N_k = |S_0|(a_1 + \dots + a_k) +
(|S_1|-|S_0|)k$. 
On the other hand, if $h$ is not of the form $N_k$,
$G^h(\tau_r(x)) = [0; B, S_*, \dots]$ with $B$ a suffix of either $S_0$
or $S_1$, $B \neq S_1$, and $S_*$ equal to either $S_0$ or $S_1$.  By
Lemma \ref{prosuf} it follows immediately that 
$$G^h(\tau_r(x)) \geq [0; S_1] \geq \tau_r(y).$$
The last inequality above is due to the fact that $\tau_r(y) \leq \tau_r(1) = [0; S_1]$.
\end{proof}

Now we can be even more precise and characterize the action of tuning on $\EE$ and $\BB(t)$ in the following way.

\begin{proposition} \label{incl_tuned}
Let $r \in \mathbb{Q}\cap (0,1)$ be extremal, $W_r = [\omega, \alpha)$.
Then, for each $x \in [0, 1]\setminus \mathbb{Q}$ we have
\begin{enumerate}
\item $$\mathcal{E}(\tau_r(x)) = \mathcal{E}(\alpha) \cup \tau_r(\mathcal{E}(x))$$
\item 
$$\mathcal{B}(\tau_r(x)) = \mathcal{B}(\alpha) \cup \bigcup_{S \in \Sigma_\alpha}  S \cdot \tau_r(\mathcal{B}(x))$$
where $\Sigma_\alpha := \{ SS_0^n \ : \ s = [0; S] \textup{ is an }\alpha-\textup{label}, n \in \mathbb{N} \}$ is a countable set of strings.
\end{enumerate}
\end{proposition}

\begin{proof}
1. Let $y \in \mathcal{E}(\tau_r(x))$. Either $y \in \mathcal{E} \cap [\alpha, 1] = \mathcal{E}(\alpha)$, or $y \in \mathcal{E} \cap 
[\tau_r(x), \alpha)$, hence $y \leq [0; \overline{S_1}]$. In the second case, $y \in \mathcal{B}(\omega)$ hence, by Lemma \ref{concat}, 
the c.f. expansion of $y$ is a concatenation of the strings $S_0$ and $S_1$ and begins with $S_1$, therefore $y = \tau_r(v)$ for some 
$v \in [0, 1]\setminus \mathbb{Q}$. Since $G^n(y) \geq y \geq \tau_r(x)$ for all $n$, 
then by Lemma \ref{induced_order} $G^n(v) \geq v \geq x$, so $v \in \mathcal{E}(x)$. Viceversa, if $v \in \mathcal{E}(x)$, then
$G^n(v) \geq v \geq x$ for all $n$, hence, by Lemma \ref{induced_order}, $\tau_r(v) \in \mathcal{E}(\tau_r(x))$.

2. Let us assume that
$$\textup{(a) }y \in \BB(\tau_r(x))\quad  \textup{ but }\quad \textup{(b) } y \notin \BB(\alpha);$$
from (b) it follows that $y$ is in some $\alpha$-gap, therefore
$y=S\cdot y'$ for some $\alpha$-label $s=[0;S]$ and $y'<\alpha$. On
the other hand, by (a) it follows that
$G^{|S|}(y)=y'\in \BB(\tau_r(x))\cap [\omega,\alpha]$,
therefore by Lemma \ref{concat} $y'=[0;Y]$ with $Y$ an
infinite concatenation of the words $S_0, S_1$ and hence
$y' = S_0^n\cdot y''$ for some $n \in \NN$ and $y'' = \tau_r(v)$
for some $v \in (0, 1) \cap (\mathbb{R} \setminus \mathbb{Q}) \cup \{0 \}$.
By Lemma \ref{induced_order}, $v \in \BB(x)$, hence $y = SS_0^n\cdot y''$ with $y'' \in \tau_r \BB(x)$.

On the other hand, suppose $y \in S S_0^n \cdot \tau_r \BB(x)$ with $s=[0;S]$ an $\alpha$-label, and let $k \geq 0$.
Then for $k < |S|$ we have that $G^k(S\cdot y)$ lies between $G^k(s)$ and $G^k(S\cdot \alpha)$, which are both above $\alpha$.
For $|S| \leq k < |S| + n |S_0|$, the c.f. expansion of $G^k(y)$ starts with a prefix of $S_0$, hence it is larger than $\alpha$ by Lemma \ref{prosuf}.
Finally, for $k \geq |S| + n |S_0|$ we have $G^k(y) = G^{h}(\tau_r (z))$ with $h = k - |S| + n |S_0|$ and $z \in \BB(x)$, hence by Lemma \ref{induced_order}  
$$ G^k(y) \geq \tau_r(x).$$
The three cases imply that $y \in \BB(\tau_r(x))$. 

Finally, $\tau_r(x) \leq \alpha$ implies $\BB(\alpha) \subseteq \BB(\tau_r (x))$, completing the proof.
\end{proof}

\begin{corollary} \label{dimtuned}
Let $r \in \mathbb{Q} \cap (0,1)$ be extremal, $W_r = [\omega, \alpha)$, and $x \in [0, 1] \setminus \mathbb{Q}$. Then we have
\begin{equation} \label{inclB}
\textup{H.dim }\mathcal{B}(\tau_r(x)) = \max\{ \textup{H.dim }\mathcal{B}(\alpha), \textup{H.dim }\tau_r(\mathcal{B}(x)) \}
 \end{equation}
 and
\begin{equation} \label{inclE}
\textup{H.dim }\mathcal{E}(\tau_r(x)) = \max\{ \textup{H.dim }\mathcal{E}(\alpha), \textup{H.dim }\tau_r(\mathcal{E}(x)) \}.
 \end{equation}
\end{corollary}

\begin{remark}
We will actually see in Section \ref{dimension} that it is possible to determine precisely which term attains the maximum, i.e.
$$\textup{H.dim } \EE(\tau_r(x)) = \left\{ \begin{array}{ll} \textup{H.dim } \tau_r\EE(x) & \textup{if }\alpha > c_F \\
                                             \textup{H.dim } \EE(\alpha) & \textup{if } \alpha < c_F 
                                            \end{array} \right.$$
and                                       
                                            $$\textup{H.dim } \BB(\tau_r(x)) = \left\{ \begin{array}{ll} \textup{H.dim } \tau_r\BB(x) & \textup{if }\alpha > c_F \\
                                             \textup{H.dim } \BB(\alpha) & \textup{if } \alpha < c_F, 
                                            \end{array} \right.$$
where $c_F$ is the fixed point of the tuning operator $\tau_{1/2}$.
\end{remark}

\subsection{Scaling and H\"older-continuity at the Feigenbaum point} \label{feigenbaum}

The tuning operator is useful to understand the scaling law of the Hausdorff dimension at the ``Feigenbaum point'' $c_F$. 
In this section, we will use it to establish Theorem \ref{th:scaling}. 

Let $\alpha:=[0;\overline{2,1}]$, $\tau := \tau_{1/2}$,
$\alpha_n:=\tau^n(\alpha)$ and $g_n:=\tau^n(g)$. 
Note that $\BB(g) = \{g \}$ is a singleton; thus, applying Proposition \ref{incl_tuned} inductively shows that $\BB(g_n)$ is countable, hence
$\textup{H.dim }\BB(g_n) = 0$. Using Corollary \ref{dimtuned} with $r = \tau^n(1)$,  we then
get that
\begin{align*}
\textup{H.dim } \BB(\alpha_n) & = \max \{ \textup{H.dim }\BB(g_n), \textup{H.dim } \tau^n(\BB(\alpha))\} \\
& = \textup{H.dim } \tau^n(\BB(\alpha)).
\end{align*}

For $\alpha=[0;\overline{2,1}]$ we know that $\BB(\alpha)=K(\{(1), (2)\})$ is a regular Cantor set, therefore all its tuned images will be regular Cantor sets as well. 
Indeed, let us set
$$ \left\{
\begin{array}{l}
Z_0:=(1)\\
Z_1:=(2) \\
Z_{n+1}:= Z_n Z_{n-1} Z_{n-1} \textup{ if } n\geq 1.
\end{array}
\right.
$$
Using  induction one can easily prove the following facts:
\begin{itemize}
\item  $Z_{n+1}=Z_nZ'_n$ where $Z'_n$ denotes the conjugate string of $Z_n$ (i.e. $[0;Z_n]=[0;Z'_n])$;  therefore, setting $q_n := q(Z_n)=q(Z'_n)$, equation  \eqref{eq:supermult} implies that 
\begin{equation}\label{quenne}
q_n^2 \leq q_{n+1} \leq 2 q_n^2;
\end{equation}
\item $\tau(Z_n)=Z_{n+1}$, hence
$\tau^n(\BB(\alpha))=K_n$  with $K_n:=K(\{Z_n, Z_{n+1}   \});$
\item $K_n$ is a disjoint union of its self similar copies: $K_n=(Z_n \cdot K_n) \sqcup (Z_{n+1} \cdot K_n)$. 
\end{itemize} 
Using this information together with equation \eqref{contraction} we get
\begin{align*}
\frac{1}{2 q_{n+1}} & \leq |f'_{Z_n}(x)| \leq \frac{1}{q_n^2}, \\
\frac{1}{(2 q_{n+1})^2} 
& \leq |f'_{Z_{n+1}}(x)|  \leq \frac{1}{q_n^4},
\end{align*}
and standard Hausdorff dimension estimates (eq. \eqref{dimbounds}) then yield
\begin{equation}\label{eq:HD}
\frac{\log \frac{\sqrt{5}+1}{2}}{\log q_{n+1} +\log 2} \leq \textup{H.dim }K_n\leq \frac{ \log \frac{\sqrt{5}+1}{2}}{\log q_{n+1} -\log 2}.
\end{equation}
On the other hand, it is easy to check that 
$$r_1:=[0;Z_2Z_1Z_2]=[0;S_1S_0S_1S_1S_0]=[0;S_1S_0S_1S_1S_1]$$ 
is the pseudocenter of $[\alpha_1, c_F]$, being a convergent of both endpoints of the interval. 
Using the tuning operator $\tau$ (which has $c_F$ as fixed point) we get that $r_n=[0;Z_{n+1}Z_nZ_{n+1}]=[0;Z_{n+1}Z_nZ'_{n+1}]$ is a common convergent of the endpoints of $[\alpha_n, c_F]$. This means that
\begin{equation}\label{eq:upperlower}
\alpha_n=(Z_{n+1}Z_nZ_{n+1})\cdot \tilde{\alpha}_n, \ \ \ c_F=(Z_{n+1}Z_nZ_{n+1}')\cdot \tilde{c}_n,
\end{equation} 
for some $\tilde{\alpha}_n, \tilde{c}_n \in K(\{(1), (2)\})$. Of course $q(Z_{n+1}Z_nZ_{n+1})=q(Z_{n+1}Z_nZ'_{n+1})$; moreover, using equation \eqref{eq:supermult} together with the estimate \eqref{quenne} we get that
\begin{equation}
\frac{1}{2}q_{n+1}^5 \leq q(Z_{n+1}Z_nZ_{n+1})^2 \leq 16 q_{n+1}^5.
\end{equation}
Since $K(\{(1), (2)\}) \subset [1/3,1]$ and $f_{Z_{n+1} Z_n Z_{n+1}}(r_n) = 0$,  equations \eqref{contraction} and \eqref{eq:upperlower} 
yield the same upper and lower estimates for both $r_n-\alpha_n$ and $c_F-r_n$, namely
$$\frac{1}{96q_{n+1}^5} \leq r_n-\alpha_n \leq \frac{2}{q_{n+1}^5}, \ \ \  \ \ \ \ \ \ \ \ 
\frac{1}{96q_{n+1}^5} \leq c_F-r_n \leq \frac{2}{q_{n+1}^5}. $$
Thus, writing $c_F-\alpha_n= (c_F-r_n)+( r_n-\alpha_n)$ we finally get 
\begin{equation}\label{eq:distance}
\frac{1}{48 q_{n+1}^5} \leq c_F-\alpha_n \leq \frac{ 4}{ q_{n+1}^5}.
\end{equation}  
 Equations \eqref{eq:HD} and \eqref{eq:distance} show the following scaling law for $\BB(t)$ at $c_F$: 
\begin{equation} \label{scaling law}
\lim_{n \to +\infty} \textup{H.dim }\BB(\alpha_n) \log\left(\frac{1}{|c_F-\alpha_n|}\right)= 5 \log \frac{\sqrt{5}+1}{2}.
\end{equation}
As a consequence, the map $\eta(t) = \textup{H.dim }\BB(t)$ is not H\"older-continuous at $c_F$.

One can use this estimate to get a result on the modulus of continuity of the dimension function $\eta$: for $ \alpha_n \leq t \leq \alpha_{n+1}$
one has $\eta(\alpha_{n+1})\leq \eta(t) \leq \eta(\alpha_n)$, and we get for $t\to c_F^-$
$$\frac{5}{2} \log \frac{\sqrt{5}+1}{2} +o(|c_F-t|) \leq \eta(t) \log\frac{1}{c_F-t} \leq 10 \log \frac{\sqrt{5}+1}{2}+o(|c_F-t|),$$
which completes the proof of Theorem \ref{th:scaling}. 

Finally, let us note that one cannot do much better than this, as the product $\eta(t) \log\frac{1}{c_F-t} $ does not converge as $t\to c_F$. Indeed, if we set 
$\beta_n:=\tau^n([0;\overline{3}])$ we can use the same method as before to check that $c_F-\beta_n$ goes to $0$ at a rate of $q_{n+1}^{-3}$ while 
$\eta(\beta_n)=\eta(\alpha_n)$, so that
$$
\lim_{n \to +\infty} \textup{H.dim }\BB(\beta_n) \log\left(\frac{1}{|c_F-\beta_n|}\right)= 3 \log \frac{\sqrt{5}+1}{2}.
$$
which is not the same scaling law as in \eqref{scaling law}. 

\begin{figure}[ht] \label{feigenpicture}
  \centering
\includegraphics[scale=1]{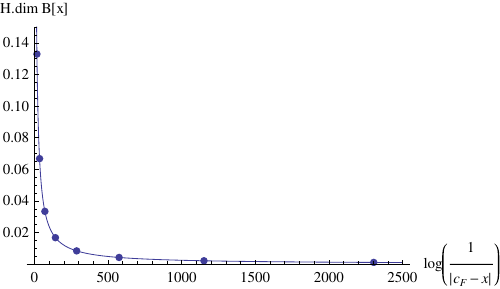}
\caption{Relation between Hausdorff dimension and distance from $c_F$. The dots correspond to successive renormalizations $\alpha_n$, $3 \leq n \leq 10$, 
while the curve is $xy = 5 \log\frac{\sqrt{5}+1}{2}$.}
\end{figure}

\section{Dominant points} \label{dominant}

\begin{definition}
A finite string $S$ of positive integers and even length is called \emph{dominant} if 
$$XY  < <   Y$$
for every splitting $S = XY$ where $X$, $Y$ are finite, nonempty strings.
\end{definition}

Let us remark that every dominant string is extremal, but not viceversa.
For instance, the strings $(5,2,4,3)$ and $(5,2,4,5)$ are both extremal, but the first is dominant while the second is not.

\begin{definition}
A point $x \in \mathcal{E}$ is \emph{dominant} if there exists a dominant word $S$ such that
$x = [0; \overline{S}]$.
\end{definition}

We define the set $UT$ of \emph{untuned parameters} as the complement in $\mathcal{E}$ of all tuning windows: 
$$UT := \mathcal{E} \setminus \bigcup_{r \in \mathbb{Q}_E} W_r.$$
The importance of dominant points lies in the fact that they approximate all untuned parameters:

\begin{proposition} \label{densitydominant}
The set of dominant points is dense in $UT \setminus \{g\}$. 
 More precisely, every parameter in $UT \setminus \{g\}$ is accumulated from the right 
by dominant parameters.

\end{proposition}

The proof of such a fact will be given in Section \ref{proofdensity}: we can now see how, by using
the density of dominant points, one can prove Proposition \ref{Lipschitz} 
stated in the introduction: 

\begin{proposition}
Given any point $x \in UT$ and any $y > x$, then $\mathcal{E}(x)$ contains a Lipschitz image of $\mathcal{B}(y)$. 
More precisely, there exists a finite string $S$ such that $S \cdot \BB(y) \subseteq \EE(x)$.
\end{proposition}

\begin{proof}
If $x = g$, then $\mathcal{B}(y) = \emptyset$ and the claim is trivial. Otherwise, by Proposition \ref{densitydominant},
 $x$ is accumulated from the right by dominant points, hence one can choose $z \in (x, y)$ a dominant point $z = [0; \overline{S}]$ 
such that $S$ is a prefix of $x$ but not of $y$. We claim that 
$$S \cdot \mathcal{B}(y) \subseteq \mathcal{E}(z).$$
Indeed, let $v \in \mathcal{B}(y)$, $w := S \cdot v$, and $k \geq 0$. Since $v \geq y \geq z$ and $S$ is even, $w = S \cdot v \geq S \cdot z = z$. 
If $k < |S|$, then there exists a decomposition $S = S'S''$ such that $G^k(w) = S'' \cdot v > S \cdot v = w$ because $S$ is dominant. On the other hand, 
if $k \geq |S|$ then $G^k(w) = G^h(v) \geq y$ for $h := k - |S|$ and $y > S \cdot v = w$ because $S$ is not a prefix of $y$, hence $G^k(w) > w$.

As a consequence, the set $S\cdot \mathcal{B}(y) \subseteq \EE(z) \subseteq \EE(x)$ is a Lipschitz image of $\mathcal{B}(y)$ inside $\EE(x)$. 
\end{proof}

\subsection{Density of dominant points} \label{proofdensity}

In order to prove Proposition \ref{densitydominant}, we will need the following definitions:
given a string $S$, the set of its \emph{prefixes-suffixes} is 
$$\begin{array}{lll} PS(S) & := & \{ Y :  Y \textup{ is both a proper prefix and a suffix of }S\} = \\
   & = & \{ Y : \ 0<|Y|<|S| , \ \exists \ X, Z\textup{ s.t. } S = XY = YZ \}.
  \end{array}$$
Moreover, we have the set of \emph{residual suffixes}
$$RS(S) := \{ Z \ : \ S = YZ, \ Y \in PS(S) \}.$$

\begin{lemma} \label{prefsuffisodd}
Let $S$ be an extremal word and $Y \in PS(S)$, then 
\begin{enumerate}
\item $|Y|<|S|-1$; 
\item
$Y$ is an extremal word of odd length. 
\end{enumerate}
\end{lemma}

\begin{proof}
(1) This claim is almost trivial, because if $|Y|=|S|-1$ then $S$ is a sequence
repeating one single figure, and since $S$ is extremal this cannot happen. 

(2) Suppose $S = XY = YZ$. Then by extremality $XY < YX$, hence $XYY <
YXY$ and, by substituting $YZ$ for $XY$, $YZY < YYZ$.  If $|Y|$ were
even, it would follow that $ZY < YZ$, which contradicts the
extremality of $S = YZ < ZY$. Hence $|Y|$ is odd.  Suppose now $Y =
AB$, with $A$ and $B$ non-empty strings. Then $S = XAB < BXA$. By
considering the first $k := |Y|$ characters on both sides of this
equation, $Y = AB = (S)_1^k \leq (BXA)_1^k = BA$, where in the last equality we used that
$S = YZ = ABZ$. In fact the
preceding inequality is strict because if $Y = AB = BA$, then $Y =
P^k$ for some string $P$ and $k\geq 2$, and this would mean that
$P^2\in PS(S)$, which is impossible since $P^2$ has even length.
\end{proof}

\begin{lemma}\label{largerthanprim}
Let $S$ be an extremal word of even length such that $\gamma = [0; \overline{S}] \in UT\setminus\{g\}$, and let $Y \in PS(S)$, $S = YZ$. 
Let $\alpha(Y)$ be the root of the maximal tuning window $W_r$ which contains $[0; Y]$. Then 
\begin{equation}\label{eq:claim}
[0; \overline{Z}] > \alpha(Y).
\end{equation}
\end{lemma}

\begin{proof}
Let $r = [0; S_0] = [0; S_1]$ be the expansions of $r$ in c.f. of even and odd length, respectively. 
Since $[0;\overline{Y}]\in \EE\cap W_r$
and $|Y|$ is odd, then $Y$ is a concatenation of the strings $S_0$ and $S_1$; we claim
\begin{equation} \label{E:concat}
\beta := [0; \overline{ZY}] > \alpha(Y) = [0; \overline{S_0}].
\end{equation}
Indeed, suppose $\beta \leq \alpha(Y)$; since the c.f. expansion of $Y \cdot \alpha(Y)$ is a concatenation of strings 
$S_0$ and $S_1$, then $Y \cdot \alpha(Y) \geq [0; S_1, \overline{S_0}] = \omega(Y)$ where $\omega(Y)$ is the lower endpoint of the tuning window $W_r$.
Hence, from $\beta \leq \alpha(Y)$, then $\gamma = Y \cdot \beta \geq Y \cdot \alpha(Y) \geq \omega(Y)$.
On the other hand, since $S = XY$ is extremal, then 
$XY < YX$, hence also 
$$S Y = XYY < YXY = YS$$
so, by the string lemma (Lemma \ref{stringlemma}), we have $\gamma = [0; \overline{S}] < [0; \overline{Y}] < \alpha(Y)$,
so that $\gamma \in W_r$, which contradicts the hypothesis that $\gamma$ is untuned. Therefore, we have proved \eqref{E:concat}.

Now, if \eqref{eq:claim} did not hold, then $[0; \overline{Z}] \leq [0;\overline{S_0}] < [0;\overline{ZY}]$, thus $Z$ should be a prefix of 
$\overline{S_0}$, since both $[0; \overline{Z}]$ and $[0;\overline{ZY}]$ have c.f. starting with $Z$, and $[0; \overline{S_0}]$ lies between them. Then we can write $Z = S_0^k V$ with $V$ prefix of $S_0$, $V \neq \emptyset$ since $|Z|$ is odd. If $S_0 \neq (1, 1)$, 
then $S_0$ is extremal and we can write $S_0 = V W$ for some string $W$. Then, since $S_0$ is extremal, we have 
$$VW < WV$$ and since $V$ has odd length this implies 
$$S_0 V = VWV < VVW = V S_0$$
and, by appending $S_0^{k}$ (which has even length) on both sides we get 
$$S_0^{k+1} V < S_0^{k} V S_0.$$
Hence, by the string lemma $[0; \overline{Z}] = [0; \overline{S_0^k V}] > [0; \overline{S_0}]$, contradiction. On the other hand, if
$S_0 = (1, 1)$ one would conclude that $Z$ is made up by 1's, which is impossible as well (since $Z\in RS(S)$, hence $Y$ would also be made up of all 1's).
\end{proof}

\begin{lemma} \label{newstringlemma}
Let $R$ be a string and $\rho=[0;\overline{R}]$ be the fixed point of the contraction $x\mapsto R \cdot x$. 
Then 
$$ R\cdot t > t  \ \ \mbox{ for all } \ 0\leq t <\rho.$$
\end{lemma}

\begin{proof}
Set $\varphi(x):= R\cdot x - x$; since  $x\mapsto R \cdot x$ is a contraction, $\varphi'(x)<0 \ \forall x\in (0,1)$ and so 
$\varphi(t)>\varphi(\rho)=0$ for all $t<\rho$.
\end{proof}

\begin{lemma}\label{inductivestep}
Let $S$ be an extremal word of even length, $S\neq (1,1)$ and $T$ be either a dominant word, or $T=(1,1)$. Suppose moreover that
\begin{enumerate}
  \item  $S  < <   T$; and 
  \item $\overline{T} < \overline{Z}$ $\ \forall Z \in RS(S)$.
\end{enumerate}
Then for any $n \geq 1$ and for $m$ sufficiently large, $S^n T^m$ is a dominant word. 
\end{lemma}

\begin{proof}
\begin{enumerate}
 \item From 1. we obtain
$$S^n T^m  < <   T^a, \quad a \geq 1$$
$$S^n T^m  < <   S^b T^m, \quad b < n.$$
\item If $S = XY$, $XY  < <   YX$ by extremality, hence
$$S^n T^m  < <   Y S^b T^m \quad \forall b \geq 1.$$
\item If  $T$ is dominant, $T  < <   U$ whenever $T = QU$, thus
$$S^n T^m  < <   T  < <   U.$$
The same conclusion holds if $T=(1,1)$, since the first digit of $S$ is strictly greater than $1$.
\item One is left to prove that $S^n T^m  < <   Y T^m$ whenever $S = XY$.
If $Y \notin PS(S)$, then $XY  < <   Y$ and the proof is complete. Otherwise, 
$S = XY = YZ$, $|Y| \equiv 1 \mod 2$ by Lemma \ref{prefsuffisodd}. Moreover, since 
$YZ < ZY$, by the string lemma (Lemma \ref{stringlemma}) we obtain $[0; \overline{ZS^{n-1}}] > [0; \overline{Z}]$, hence 2. 
implies $[0; \overline{T}] < [0; \overline{Z S^{n-1}}]$. So, setting $R=Z S^{n-1}$ in Lemma \ref{newstringlemma}, we get
$[0; Z S^{n-1} \overline{T}] > [0; \overline{T}]$, hence for $m$ large enough
$Z S^{n-1} T^m  >>  T^m$ and then 
$$S^nT^m  < <   Y T^m.$$ 
\end{enumerate}
\end{proof}

\noindent \textit{Proof of Proposition \ref{densitydominant}.}
By density of the roots of the maximal tuning windows, it is enough to prove that every $\alpha$ which is root of a maximal tuning window, 
$\alpha \neq g$, can be approximated from the right by dominant points. Hence we can assume $\alpha = [0; \overline{S}]$, 
$S$ an extremal word of even length, and $1$ is not a prefix of $S$. If $S$ is dominant, a sequence of approximating dominant points is 
given by $[0;\overline{ S^n, 1, 1}]$, $n \geq 1$. The rest of the proof is by induction on $|S|$. If $|S| = 2$, then $S$ itself is dominant and we are in the previous case.
If $|S| > 2$, either $S$ is dominant and we are done, or $PS(S) \neq \emptyset$ and also $RS(S) \neq \emptyset$. Let us choose $Z_\star \in RS(S)$ 
such that $$[0; \overline{Z_\star}] := \min \{ [0; \overline{Z}] \ : \ Z \in RS(S) \}$$
and $Y_\star \in PS(S)$ such that $S = Y_\star Z_\star$. Let $\alpha(Y_\star)$ be the root 
of the maximal tuning window $[0; Y_\star]$ belongs to. Then by Lemma  \ref{largerthanprim},  
$[0; \overline{Z_\star}] > \alpha(Y_\star)$, and by minimality
$$\alpha(Y_\star) < [0; \overline{Z}] \quad \forall Z \in RS(S).$$
If $\alpha(Y_\star)=g$ we can directly apply Lemma \ref{inductivestep} with $T=(1,1)$ to conclude that for $m$ large enough, $S^nT^m$ is a dominant word,  and $[0;S^nT^m]$ is arbitrarily close to $[0; \overline{S}]$ for sufficiently big $n$.

If not, then $g\neq\alpha(Y_\star) = [0; \overline{P_0}]$ with $|P_0|$ even. By Lemma \ref{prefsuffisodd}, $Y_\star$ is an extremal word of  odd length,  and, since $[0; \overline{Y_\star}]$ belongs to the tuning window of root $\alpha(Y_\star)$, by means of the results of Section \ref{tuning} we deduce that  $Y_\star$ has the odd prefix $P_1$ (where $P_1$ denotes the conjugate string of $P_0$, i.e. $[0;P_1]=[0;P_0]$); hence $Y_\star < <  P_0$, and a fortiori also $S < <  P_0$.

Moreover 
$$|P_0|\leq |Y_\star|+1\leq|S|$$
and actually $|Y_\star|+1<|S|$ because otherwise the first digit of $Y_\star$ would appear twice at the beginning of $S$, contradicting the fact that $S$ is extremal. Since $|P_0|<|S|$, we can use the inductive hypothesis to claim that there is $\gamma=[0;\overline{T}]$ which is dominant and

$$ \alpha(Y_\star) <  [0; \overline{T}] < [0; \overline{Z}] \quad \forall Z \in RS(S)$$
and $\gamma$ can also be chosen close enough to $\alpha(Y_\star)$ so that $P_0$ is prefix of $T$, which implies
$$S  < <   T.$$
Again, by Lemma \ref{inductivestep}, $S^nT^m$ is a dominant word for $m$ large enough, of even length if $m$ is even, and arbitrarily close to $[0; \overline{S}]$ 
as $n$ tends to infinity.
\qed

\section{Fractal dimension} \label{dimension}

\subsection{Continuity of Hausdorff dimension} \label{continuity}

Let us now prove the first part of Theorem \ref{main}, namely the continuity of the Hausdorff dimension of $\mathcal{B}(t)$. 

The continuity of the monotone map $t \mapsto \textup{H.dim } \BB(t)$ at $t=0$ follows immediately from equation \eqref{eq:zero}.
On the other hand, the continuity at any point $t_0>0$ can be obtained using a result of Raith.
Recall that a piecewise monotonic, piecewise $C^1$ map $T: [0,1] \to [0,1]$ is \emph{expanding} if there exists $n \geq 1$ such that 
$\inf |(T^n)'(x)| > 1$, where the inf is taken over all points $x$ where the derivative is defined.

Theorem 2 of \cite{Ra}  in the simplest case (i.e. setting $K=1$ in \cite[Thm 2]{Ra}, see also \cite{Ra1}) states:

\begin{theorem}\label{th:Ra}
Let $T:[0,1]\to [0,1]$ be an expanding piecewise monotonic map, and $Q$ be an open interval. Then for every $\epsilon>0$ there exists $\delta>0$ such that if $\widetilde{Q}$ is another interval $\delta$-close to $Q$ then
$$| \textup{H.dim } R(\widetilde{Q}) - \textup{H.dim } R({Q})|< \epsilon$$
where $R(Q):=\{x: T^k(x) \notin Q \ \ \forall k \geq 0\}$. 
\end{theorem}

So, in order to prove that the map $t \mapsto \textup{H.dim } \BB(t)$ is continuous at a point $t_0\in (0, 1/2)$ we proceed as follows:
\begin{enumerate}
\item Note that equation \eqref{eq:BBF} allows us to write $\BB(t)$ in terms of the Farey map
as 
$$\BB(t) = \{ x \in [0, 1] \ : \ F^k(x) \geq t \ \ \forall k \in \NN \}.$$ 
\item Let us call $T$ the even map obtained by ``rectifying" the graph of $F$ near the endpoints 0 and 1: namely fix $\delta = \frac{t_0/2}{1+t_0/2}$ and let  $T$ be defined as 
$$T(x) := \left\{ \begin{array}{ll} 
\frac{x}{1-\delta} & \textup{if }0 \leq x \leq \delta \\
F(x) & \textup{if } \delta < x \leq 1 - \delta \\
\frac{1-x}{1 - \delta} &  \textup{if } 1- \delta < x \leq 1.
\end{array} \right.
$$
The map $T$ is continuous and expanding (even if $F$ is not). Moreover for all $t \in (t_0/2, 1/2)$ we have that
\begin{equation}\label{eq:BB=BBT}\begin{split}
\widehat{\BB}(t)&:=  \{ x \in [0, 1] \ : \ T^k(x) \geq t \ \ \forall k \in \NN \}=\\
 &=\{ x \in [0, 1] \ : \ F^k(x) \geq t \ \ \forall k \in \NN \} =\BB(t); 
 \end{split} 
\end{equation}
this is so because for $t \in (t_0/2, 1/2)$ the set  $\BB(t)$ (and hence also $\widehat{\BB}(t)$) is contained in the interval $ [t,1/(1+t)]\subset [\delta, 1-\delta]$, where the maps $F$ and $T$ coincide.
\item We can now apply Theorem \ref{th:Ra} to the map $T$ defined above using $Q=(0,t_0)$ and $\widetilde{Q} = (0,t)$ and we get that for all $\epsilon>0$ there is $\delta>0$ such that if $|t-t_0|<\delta $ then $| \textup{H.dim } \widehat{\BB}(t) - \textup{H.dim } \widehat{\BB}(t_0)|< \epsilon$.
\end{enumerate}
This concludes the proof of Theorem \ref{main}(1). \qed

\subsection{The local dimension of $\mathcal{E}$} \label{S:local}

We are now ready to prove Theorem \ref{main}(2), namely that, for each $t \in [0, 1]$, 
$$\textup{H.dim }\BB(t) = \textup{H.dim }\EE(t)$$
where $\EE(t):= \EE \cap [t,1]$. 
Let us recall that by definition $\EE(x) \subseteq \BB(x)$, hence one inequality is trivial.
Moreover, since both functions $x \mapsto \BB(x)$ and $x \mapsto \EE(x)$ are locally constant
around any $x$ which does not belong to $\mathcal{E}$, it is enough to prove that
\begin{equation} \label{equalityonE}
\textup{H.dim }\BB(x) = \textup{H.dim }\EE(x) \quad \textup{for all }x \in \mathcal{E}.
\end{equation}

We use the following lemma.

\begin{lemma} \label{Lipschitz_tuning}
For each $n \geq 0$, $t > 0$ and $r \in \mathbb{Q} \cap (0,1)$ extremal, the map $\tau^n \tau_r \tau^{-(n+1)} : \tau^{n+1}\BB(t) \rightarrow \tau^n\tau_r \BB(t)$ is 
Lipschitz continuous.  
\end{lemma}

\begin{proof}
Let us define the partial order $\leq_\star$ on the set of finite strings of positive integers: given 
two strings $A = (a_1, \dots, a_n)$ and $B = (b_1, \dots, b_m)$, 
$A \leq_\star B$ if there exists a strictly increasing function $K : \{1, \dots, n\} \mapsto \{1, \dots, m \}$ such that 
for all $1 \leq i \leq n$, $a_i \leq b_{K(i)}$. Notice that $\leq_\star$ behaves well under concatenation ($A \leq_\star B, C \leq_\star D \Rightarrow AC \leq_\star BD$)
and tuning (i.e. $A \leq_\star B \Rightarrow \tau_r(A) \leq_\star \tau_r(B)$ for any $r$). Moreover, if $A \leq_\star B$, then
$q(A) \leq q(B)$. 
Let us now consider $x, y \in \BB(t)$. There exists a string $S$ such that $x = S \cdot x', y = S \cdot y'$ with $x', y'$ without any common 
digit at the beginning of their continued fraction expansion. Notice that since $x', y' \in \BB(t)$, there exists $C > 0$ (which depends on $t$ and $n$)
such that $d(\tau^{n+1}(x'), \tau^{n+1}(y')) \geq C$. Let us denote by $\Sigma_1$ and $\Sigma_2$ the two finite strings 
$\Sigma_1 := \tau^n \tau_r(S)$, $\Sigma_2 := \tau^{n+1}(S)$, and $S_1$, $S_0$ be the even and odd string representing $r$. 
Let us notice that by eq. \eqref{contraction}
$$d(\tau^n \tau_r(x), \tau^n \tau_r(y)) \leq 
\sup |f_{\Sigma_1}'| d(\tau^n\tau_r(x'), \tau^n \tau_r(y')) \leq \frac{1}{q({\Sigma_1})^2},$$
$$d(\tau^{n+1}(x), \tau^{n+1}(y)) \geq 
\inf |f_{\Sigma_2}'| d(\tau^{n+1}(x'), \tau^{n+1} (y')) \geq \frac{C}{4q({\Sigma_2})^2}, $$
hence
$$ d(\tau^n \tau_r(x), \tau^n \tau_r(y)) \leq \frac{4 q(\Sigma_2)^2}{q(\Sigma_1)^2 C} d(\tau^{n+1}(x), \tau^{n+1}(y)).$$ 
Now, since $(1, 1) \leq_\star S_0$ and $(2) \leq_\star S_1$, then by concatenation $\tau(S) \leq_\star \tau_r(S)$, and by tuning
$\Sigma_2 = \tau^{n+1}(S) \leq_\star \tau^n \tau_r (S) = \Sigma_1$, hence $q(\Sigma_2) \leq q(\Sigma_1)$ and the claim follows. 

\end{proof} 

Let now $\tau = \tau_{1/2}$ be the tuning operator corresponding to $\frac{1}{2}$. 
Note that $W_{1/2} = [g^2, g)$ where $g$ is the golden 
mean, and $\tau$ has a unique fixed point $c_{F} \in [g^2, g)$. By the estimates of Section \ref{feigenbaum}, 
$\textup{H.dim }\BB(c_F) = 0$, and \eqref{equalityonE} holds for $x \geq c_F$.

Given any $x \in \mathcal{E} \cap (0, c_F)$, then there exists 
$n \geq 0$ and $y \in \mathcal{E} \setminus [g^2, g)$ such that $x = \tau^n (y)$. 
Now, by Corollary \ref{dimtuned} and the fact $\mathcal{E}(g) = \mathcal{B}(g) = \{g\}$ is a single point, 
\begin{equation} \label{E:tau-dim}
\textup{H.dim }\mathcal{B}(\tau^n(y)) = \textup{H.dim } \tau^n \mathcal{B}(y), 
\end{equation}
\begin{equation}
\textup{H.dim }\mathcal{E}(\tau^n(y)) = \textup{H.dim } \tau^n \mathcal{E}(y).
\end{equation}
We have two cases:
\begin{enumerate}
 \item If $y \in UT$, then by Proposition \ref{Lipschitz} for each $z > y$ there exists a finite string $S$ such that $S \cdot \BB(z) \subseteq \EE(y)$, 
hence by applying $\tau^n$ to both sides
$$\tau^n(S) \cdot \tau^n(\BB(z)) = \tau^n( S \cdot \BB(z)) \subseteq \tau^n(\EE(y)),$$
hence $\tau^n(\EE(y))$ contains a Lipschitz image of $\tau^n(B(z))$ and 
$$ \sup_{z > y} \textup{H.dim }\BB(\tau^n(z)) = \sup_{z > y} \textup{H.dim } \tau^n \BB(z) \leq \textup{H.dim } \tau^n \EE(y)) = \textup{H.dim } \EE(\tau^n (y)).$$
Moreover, by continuity of Hausdorff dimension and of $\tau$,  
$$\sup_{z >y} \textup{H.dim } \BB(\tau^n (z)) = \textup{H.dim }\BB(\tau^n (y))$$
hence $\textup{H.dim }\BB(x) \leq \textup{H.dim }\EE(x)$ and \eqref{equalityonE} follows.
\item If $y \notin UT$, then $y = \tau_r(z)$ for some extremal rational $r < g^2$ and some $z \in [0, 1] \setminus \mathbb{Q}$.
Now, applying the tuning operator to both sides of formula 1. from Proposition \ref{incl_tuned} 
\begin{equation} \label{E:tuning1}
\textup{H.dim }\tau^n \EE(\tau_r(z))= \max \{ \textup{H.dim } \tau^n \tau_r \EE(z), \textup{H.dim }\tau^n \EE(\alpha_r) \},
\end{equation}
where $\alpha_r$ is the root of the tuning window $W_r$ relative to $\tau_r$, and the same formula with $\mathcal{E}$ replaced by $\mathcal{B}$ holds
by formula 2. of Proposition \ref{incl_tuned}. Now, since $\alpha_r < g^2$,
$$\tau^n \EE(\alpha_r) \supseteq \tau^n \EE(g^2) = \tau^n \EE(\tau(0)) \supseteq \tau^{n+1} \EE,$$
while 
$$\tau^n\tau_r \EE(z) \subseteq \tau^n \tau_r \EE,$$
and since the map $\tau^n \tau_r \tau^{-(n+1)}$ restricted to numbers of bounded type is Lipschitz (Lemma \ref{Lipschitz_tuning}), then 
$\textup{H.dim } \tau^n \tau_r \EE \leq \textup{H.dim }\tau^{n+1}\EE$.
Therefore
$$\textup{H.dim } \tau^n\tau_r \EE(z) \leq \textup{H.dim }\tau ^n \tau_r\EE \leq \textup{H.dim } \tau^{n+1} \EE \leq \textup{H.dim }\tau^n\EE(\alpha_r),$$
hence $\textup{H.dim }\tau^n \EE(\tau_r (z)) = \textup{H.dim } \tau^n\EE(\alpha_r)$ and similarly 
\begin{equation} \label{E:tuned-dim}
\textup{H.dim }\tau^n \BB(\tau_r (z)) = \textup{H.dim } \tau^n\BB(\alpha_r).
\end{equation}
Since $\alpha_r \in UT$, by case 1. we obtain
$\textup{H.dim }\EE(\tau^n (\alpha_r)) = \textup{H.dim } \BB(\tau^n (\alpha_r))$, hence
\begin{align*}
\textup{H.dim }\EE(\tau^n \tau_r (z)) & = \textup{H.dim }\tau^n \EE(\tau_r (z))  = \textup{H.dim } \tau^n \EE(\alpha_r) \\
& = \textup{H.dim } \EE(\tau^n (\alpha_r))  = \textup{H.dim } \BB(\tau^n (\alpha_r)) \\ 
& = \textup{H.dim } \tau^n \BB(\alpha_r)  = \textup{H.dim } \tau^n \BB(\tau_r (z)) \\
& = \textup{H.dim } \BB(\tau^n \tau_r (z)),
\end{align*}
which completes the proof of Theorem \ref{main}(2).
\end{enumerate}
\qed 

Now let us turn to the proof of Theorem \ref{main} (3). 
Let $r \in \mathbb{Q}_E \cap (0, c_F)$. Then there exists $n \geq 0$ such that $r = \tau^n(s)$ with $s < g^2$ extremal. 
Hence, by eq. \eqref{E:tau-dim} and \eqref{E:tuned-dim}, for any $z \in \EE$ one has 
$$\textup{H.dim }\BB(\tau^n \tau_s (z)) = \textup{H.dim }\tau^n \BB(\tau_s (z)) = \textup{H.dim } \tau^n \BB(\alpha_s) = \textup{H.dim } \BB(\tau^n (\alpha_s)).$$
Since $\tau^n \tau_s(z) = \tau_r(z)$ and $\tau^n(\alpha_s) = \alpha_r$, the above equation implies 
$$\textup{H.dim }\BB(\tau_r (z)) = \textup{H.dim } \BB(\alpha_r)$$
for any $z \in \mathcal{E}$, hence the Hausdorff dimension of $\BB(t)$ is constant as $t \in W_r$. This concludes the proof of Theorem \ref{main}.

\subsection{A problem about Sturmian sequences} \label{Cassaigne}

Let $\alpha \in (0, 1) \setminus\QQ$. Then a  {\it Sturmian sequence} of slope $\alpha$ is a binary sequence of the type
\begin{equation}\label{eq:sturm}
 S_{\alpha, \beta}= \{ \lfloor \alpha (n+1)+\beta \rfloor- \lfloor \alpha n+\beta \rfloor \ : \ n \in \NN \}
\end{equation}
or 
\begin{equation}
S_{\alpha, \beta}= \{ \lceil \alpha (n+1)+\beta \rceil- \lceil \alpha n+\beta \rceil \ : \ n \in \NN \}
\end{equation}
where $\beta$ is some real value. Sturmian sequences have also a geometrical interpretation: they can be 
 viewed as cutting sequences of half-lines on the plane with respect to the integral lattice $\ZZ^2$.

Given a sequence $X$ (finite or infinite) and a positive integer $m$, the set of $m$-factors of $X$ 
is the set of substrings of $X$ of length $m$:
$$\FF_m(X):=\{S = (x_{n+1}, \dots, x_{n+m})\ :  0 \leq n < |X|-m+1  \}.$$

The {\it recurrence function} of a binary sequence $X\in \{0,1\}^\NN$
is the function $R_X: \NN \to \NN\cup\{+\infty\}$ defined
by\footnote{We follow the usual convention $\inf \emptyset= +\infty$.}
$$R_X(n):=\inf\{ m \in \NN \ : \ \forall S\in \FF_m(X), \ \ \FF_n(S)=\FF_n(X)\}$$ 
while the {\it recurrence quotient} of $X$ is the maximal linear growth rate of $R_X(n)$:

$$R_X:=\limsup_{n \to +\infty}\frac{R_X(n)}{n}.$$

It is well known that the recurrence quotient of a Sturmian sequence $S_{\alpha, \beta}$
depends only on the continued fraction expansion of its slope $\alpha=[a_0;a_1,a_2,...]$.
In fact, the following formula holds (\cite{Ca}, Corollary 1)
$$R_{S_{\alpha, \beta}} = \rho(\alpha) := 2+\limsup_{k\to +\infty} [a_k;
  a_{k-1},a_{k-2},...,a_1].$$

So, if $\limsup a_k = N$, then $\rho(\alpha)\in (N+2,N+3)$; if otherwise $\alpha$ has
unbounded partial quotients, then $\rho(\alpha)=+\infty$.
The {\it recurrence spectrum of Sturmian sequences} is defined
by $$\RRR:=\{\rho(\alpha) , \ \alpha \in \RR\setminus \QQ\};$$ the
main result of Cassaigne is a characterization of $\RRR$ that, in
terms of $\EE$, can be written as follows (c.f. \cite{Ca}, Theorem 1):
\begin{equation} \label{Cassaigne_equal}
\RRR= \left\{2+\frac{1}{x} \ : \ x\in \EE \right\}.
\end{equation}

Cassaigne \cite{Ca} posed the question of determining the Hausdorff dimension of the sections
$\RRR \cap [N+2, N+3]$ for each positive integer $N$: the answer to this question is a corollary of Theorem \ref{main}, in fact:
\begin{align*}
\textup{H.dim } \RRR\cap [N+2,N+3] & = 
\textup{H.dim } \EE \cap \left[ \frac{1}{N+1},\frac{1}{N} \right] = \\
& = \textup{H.dim }\EE\left(\frac{1}{N+1}\right)= \\
& = \textup{H.dim } \BB_N.
\end{align*}
Let us recall that a precise asymptotic estimate \cite{He3} is also available for
$\textup{H.dim }  \BB_N$ (and hence for $\textup{H.dim } \RRR\cap [N+2,N+3]$), namely
\begin{equation}\label{eq:zero}
\textup{H.dim } \BB_N=1- \frac{6}{\pi^2 N} - \frac{ 72 \log N}{\pi^4 N^2} +O\left(\frac{1}{N^2}\right) 
\end{equation}
as $N \to +\infty.$ Higher order terms in the asymptotic expansion have been obtained recently in \cite{DFSU}.

\end{document}